\documentclass{article}

\usepackage{graphics,graphicx,subfig,color}
\usepackage{amsmath,amsfonts,amscd,amssymb,bm}
\usepackage{mathrsfs}
\usepackage{mathtools,eqparbox}
\usepackage{fancyhdr}
\usepackage{listings}
\usepackage{xcolor}
\usepackage{lingmacros}
\usepackage{bbm} 
\usepackage{blindtext}
\usepackage{tree-dvips}
\usepackage{empheq}
\usepackage{enumerate}
\usepackage{array}
\usepackage[section]{placeins}
\usepackage{marginnote}
\usepackage{float}
\usepackage{textcomp}
\usepackage[makeroom]{cancel}
\usepackage{verbatim}
\usepackage{geometry}
\numberwithin{equation}{section}

\usepackage{amsthm}
\theoremstyle{definition}
\newtheorem{definition}{Definition}[section] 

\theoremstyle{plain}
\newtheorem{theorem}[definition]{Theorem}
\theoremstyle{remark} 
\newtheorem{remark}[definition]{Remark}
\theoremstyle{definition}
\newtheorem{assumption}[definition]{Assumption}
\usepackage{bm}
\usepackage{comment}
\excludecomment{confidential} 
\usepackage{cancel}
\usepackage{centernot}

\DeclareMathOperator{\tr}{tr}

\DeclareMathOperator{\sym}{sym}

\newcommand{\dt}{\mathrm{d}t}

\newcommand{\D}{\mathrm{d}}

\newcommand{\bn}{\boldsymbol{n}}

\usepackage{xfrac}

\newcommand{\vertiii}[1]{{\left\vert\kern-0.25ex\left\vert\kern-0.25ex\left\vert #1\right\vert\kern-0.25ex\right\vert\kern-0.25ex\right\vert}}

\newcommand{\eps}{\varepsilon}

\newcommand{\xb}{{\boldsymbol {x}}}

\newcommand{\bC}{\boldsymbol  C}

\newcommand{\cE}{{\mathcal{E}}}

\newcommand{\bb}{\boldsymbol B}
\newcommand{\bSigma}{\boldsymbol{\Sigma}}
\newcommand{\bu}{\mathbf u}

\newcommand{\bfv}{\boldsymbol v}

\newcommand{\bfx}{\boldsymbol x}
\newcommand{\bfz}{\boldsymbol z}
\newcommand{\bff}{\boldsymbol f}

\newcommand{\bfB}{\boldsymbol B}
\newcommand{\bfA}{\boldsymbol A}

\newcommand{\bF}{\mathbf F}
\newcommand{\bD}{\mathbf D}

\newcommand{\bfsigma}{\boldsymbol{\sigma}}

\allowdisplaybreaks

\title{The diffuse interface approximation 
to
fluid-structure interaction}

\author{Francis R.~A.~Aznaran\thanks{Department of Applied and Computational Mathematics and Statistics, University of Notre Dame, USA, \texttt{faznaran@nd.edu}}
\and Martina Buka\v{c}\thanks{Department of Applied and Computational Mathematics and Statistics, University of Notre Dame, USA, \texttt{mbukac@nd.edu}}
\and Boris Muha\thanks{Department of Mathematics, Faculty of Science, University of Zagreb, Croatia, \texttt{borism@math.hr}}
}

\begin{document}

\maketitle

\begin{abstract}
We consider a fluid-structure interaction problem in the Eulerian, phase-field formulation. The problem is described using the Navier-Stokes equations for a viscous, incompressible fluid,  coupled with the incompressible hyperelasticity system, both written in the Eulerian coordinates. This allows the problem to be written in a unified formulation, using a single field for the fluid and structure velocities. To track the position of the domain, we use a phase-field approach,  resulting in a coupled Cahn-Hilliard-Navier-Stokes-type of problem for the diffuse interface fluid-structure interaction.  Under certain assumptions, we prove the convergence of the diffuse interface model  to the sharp interface fluid-structure interaction problem. To solve the problem numerically, we propose a novel, strongly coupled, second-order partitioned computational method where the system is decoupled into the Cahn-Hilliard problem, the transport problem for the left Cauchy-Green deformation tensor, and the Navier-Stokes problem. The  problems are solved iteratively until convergence at each time step. The performance of the method is illustrated on two computational examples. 

\end{abstract}

\section{Introduction}

Fluid-structure interaction (FSI) occurs in many applications, such as aerodynamics, hemodynamics, and biomedical engineering. It is a moving domain, multiphysics problem characterized by strong, nonlinear coupling, where the current domain is one of the unknowns of the problem. The FSI problems have been classically approximated using sharp interface methods, where the mesh nodes are aligned with the interface~\cite{badia2008modular,gee2011truly,serino2019stable,burman2022stability,gigante2021stability,bukavc2023time,bukavc2021refactorization}. In that context, the problem is commonly formulated in the Arbitrary Lagrangian-Eulerian (ALE) formulation.

Fully Eulerian formulations have also been used for FSI, with the advantage that they naturally accommodate large structural deformations and avoid the mesh distortion issues inherent in ALE approaches. Eulerian FSI formulations have been studied in~\cite{liu2001eulerian,dunne2006eulerian,laadhari2013fully,richter2013fully,wick2013fully,valkov2015eulerian,rycroft2020reference,nishiguchi2019full}. However, the main difficulty is that some tracking of the domain is still necessary. This has been done using the Initial Point Set method~\cite{dunne2006eulerian,richter2013fully}, where material points are tracked in a Lagrangian manner to reconstruct the domain occupied by the structure, or through level-set methods that implicitly represent the interface using the Reference Map Technique~\cite{valkov2015eulerian,rycroft2020reference}. We also mention  Eulerian formulations in the context of the Oldroyd-B model studied in~\cite{garcke2022viscoelastic,garcke2026parametric,garcke2024approximation,malek2018thermodynamics,Lozinski2003,BarrettLuSuli2017}.

To track the fluid and structure domains, in this work, we consider a phase-field approach based on the diffuse interface method.  This approach uses an  indicator function $\phi$, 
which is initially set equal to $1$ in the fluid subdomain
and $-1$ in the structure subdomain.
This phase-field  transitions rapidly but continuously between the subdomains in a  `diffuse interface' layer of width $\mathcal{O}(\eps)$.
Using the phase-field approach, the mesh nodes do not have to be aligned with the interface, which is well-suited for problems undergoing large topological changes. Additionally, the diffuse interface formulation naturally handles contact problems without requiring explicit contact detection algorithms or imposing contact constraints. This makes the phase-field framework particularly attractive for simulating scenarios involving interacting deformable bodies in fluid flow, such as cell aggregation, particle-laden flows, or valve dynamics, where traditional sharp interface methods would require complex and computationally expensive contact resolution procedures.

The diffuse interface approach has previously been used to model FSI in~\cite{Mokbel2018,Mao2023,rath2023interface,valizadeh2025monolithic,sun2014full,mao20243d,rath2024efficient}. A thermodynamically consistent phase-field model for FSI was derived in~\cite{Mokbel2018}, together with a formal sharp interface limit showing convergence of the derived equations to a traditional FSI formulation. An interface and geometry preserving method for Eulerian, phase-field fluid-structure interaction was presented in~\cite{Mao2023,rath2023interface} and demonstrated using numerical examples. Monolithic numerical methods for the Eulerian phase-field FSI problem have been presented in~\cite{valizadeh2025monolithic,sun2014full}, while the work in~\cite{valizadeh2025monolithic} also considered multi-body contact problems. Recent work in~\cite{mao20243d,rath2024efficient} considered extensions of their previous work to 3D problems and contact dynamics.

We consider the FSI between an incompressible, viscous fluid and an incompressible, hyperelastic structure. The structure is described using the  incompressible
neo-Hookean model, and the system is written in a unified Eulerian formulation using single velocity and pressure fields, together with an equation describing the  transport of the left Cauchy–Green deformation tensor. The problem is fully coupled to the Cahn-Hilliard model for the phase-field. This work is focused on the modeling error analysis. We show that  the modeling error measured in the energy norm converges to zero assuming that the phase-field  converges to the characteristic function
of the fluid domain. We also propose a partitioned, strongly-coupled numerical method for this problem where the Cahn-Hilliard equations, the transport of the  left Cauchy-Green deformation tensor, and the Navier-Stokes euqations are all solved separately. Within each time step, the three problems are subiterated until convergence.  The proposed method is used to study the numerical rates of convergence and to demonstrate the method's potential on a contact problem. 

The rest of this paper is organized as follows. Section 2 presents the sharp and diffuse mathematical models for FSI. Modeling error analysis is preformed in Section 3, and the numerical method is presented in Section 4. The numerical results are presented in Section 5, and Section 6 summarizes the key findings and draws conclusions.

\section{Mathematical model}

Assume $\hat{\Omega}_S, \hat{\Omega}_F \subset \mathbb{R}^d, d \in \{2, 3\}$, are open bounded 
disjoint
domains representing the reference structure and fluid regions, respectively, with Lipschitz boundaries.
We denote $\hat{\Omega} = \hat{\Omega}_S \cup \hat{\Omega}_F .$
Let $\hat{\xb}$ denote the 
coordinates 
in the reference configuration $\hat{\Omega}$, and $\xb$ denote the 
coordinates
in the physical configuration $\Omega$. Hats will in general denote variables 
and operators 
defined in the reference domain. While the solid domain moves in time, we assume that the 
overall 
domain $\Omega$ remains fixed, i.e., $\Omega(t) = \hat{\Omega}~\forall~t\geq 0$ (see Figure~\ref{domain}).
We first present the structure problem in the Lagrangian and Eulerian formulations, and then write the coupled FSI problem in the Eulerian form.
\begin{figure}[H]
    \centering
    \includegraphics[width = 0.50\textwidth]{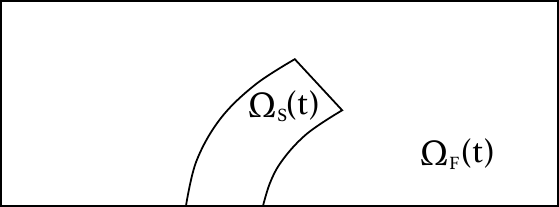}
    \caption{An example of fluid and structure domains at time $t$.}
    \label{domain}
\end{figure}

\subsection{Elastodynamics problem in the Lagrangian and Eulerian form}\label{solid}

The elasticity equation in the {Lagrangian formulation} is given as~\cite[p.~37]{Richter2017}:
\begin{align}
    \hat{J} \hat{\rho} \partial_{tt} \hat{\bu} = \hat{\nabla} \cdot \hat{\boldsymbol P} + \hat{J}\hat{\rho}\hat{\bff},
\end{align}
where $\hat{\bu}$ is the displacement, $\hat{\rho}$ is the density, $\hat{\boldsymbol P} = \hat{\bF} \hat{\Sigma}$ is the first Piola--Kirchhoff stress, 
$\hat{\bff}$ is a volume force, $\hat{\bF}$ is the deformation gradient defined by 
$
\hat{\bF} = \mathbf{I} + \hat{\nabla} \hat{\bu} (\hat{\xb}, t) 
$
\cite[Def.~2.1]{Richter2017},
$\hat{J} = \det(\hat{\bF})$, and $\hat{\bSigma}$ is the second Piola--Kirchhoff stress tensor. We note that  when computing gradients, we use  the convention $(\nabla\bfv)_{ij} = \frac{\partial v_i}{\partial x_j}$.

Using $\hat{\nabla} \hat{\boldsymbol v} = \nabla \boldsymbol v \hat{\bF}$~\cite[eq.~(2.16)]{Richter2017}, we have
\begin{gather}\label{dtF}
    D_t \hat{\bF} = D_t \hat{\nabla} \hat{\bu} = \hat{\nabla} \hat{\bfv} = \nabla{\bfv} \hat{\bF},
\end{gather}
where $D_t = \partial_t + \bfv \cdot \nabla$ denotes the material derivative, $\bfv$ is the structure velocity, and $\hat{\bfv} = D_t \hat{\bu}$~\cite{Mokbel2018,Mao2023}.
The inverse of $\hat{\bF}$ is defined as~\cite{Richter2017}
\[
    \bF = \hat{\bF}^{-1} := \mathbf{I} - \nabla \bu,
\]
and the left Cauchy--Green tensor is given by
\begin{align*}
    & 
    \bb^{-1} = \hat{\bF}^{-\top} \hat{\bF}^{-1} = \bF^{\top} \bF,
    \text{ so that }
    \bb = \hat{\bF}\hat{\bF}^{\top} = \bF^{-1} \bF^{-\top}.
\end{align*}
Assuming that the material is hyperelastic, the second Piola--Kirchhoff stress tensor can be computed as
\begin{align}
    \hat{\bSigma} =  \hat{\bF}^{-1}  \partial_{\hat{\bF}} \hat{\Psi}= 2 \partial_{\hat{\boldsymbol C}} \hat{\Psi},
\notag
\end{align}
where $\hat{\Psi}$ is the  strain energy density function,  and $\hat{\bC} = \hat{\bF}	^{\top} \hat{\bF}$ is the right Cauchy--Green deformation tensor~\cite{Richter2017}.
For an incompressible neo-Hookean constitutive model~\cite{Jain2017}, the strain energy density function is given as
\[
    \hat{\Psi} (\hat{\bC}) = \frac{G_s}{2}(\tr(\hat{\bC}) - d),
\]
where $G_S$ is the shear modulus. 
Note that $ \tr(\hat{\bC}) = \tr(\hat{\bF}^{\top}\hat{\bF}) = |\hat{\bF}|^2$,
and that 
$ \partial_{\hat{\bF}} \hat{\Psi} (\hat{\bF}) = G_S \hat{\bF}.
$ Therefore, we have
\begin{gather*}
\hat{\boldsymbol P} = \hat{\bF} \hat{\bSigma} = G_S \hat{\bF}.
\end{gather*}

In the Eulerian formulation, the elastodynamics equation is given by~\cite{Richter2017}
\begin{align}
    {\rho} D_{t} {\bfv} = {\nabla} \cdot {\bfsigma} + \rho\bff,
\end{align}
where $\bfv$ is the solid velocity, and $\bfsigma$ is the Cauchy stress tensor given by
$\bfsigma = J \bF^{-1} \bSigma \bF^{-\top}$. 
We define $\bSigma$ in the Eulerian coordinates as
\begin{align*}
    \bSigma := \hat{\bSigma} = G_S \mathbf{I}. 
\end{align*}
Taking into account that the model is incompressible, we can write the Cauchy stress tensor as
\begin{align*}
    \bfsigma = G_S \bF^{-1} \bF^{-\top} = G_S \bb. 
\end{align*}

\subsection{The sharp interface FSI model in the Eulerian form}

The coupled FSI problem in the Eulerian form is given as follows~\cite{Mokbel2018}:
\begin{align}
& 
D_t \left( \rho \bfv \right)   
-\nabla \cdot \left( 2\mu \bD(\bfv)  + G \bb \right)
 +\nabla p=  \rho \boldsymbol f & \textrm{in } \;\ \Omega \times (0, T),
\label{s1}
\\
& \nabla \cdot \bfv=0  & \textrm{in } \;\ \Omega \times (0, T),
\label{s2}
\\
    & G\left( D_t \bfB - (\nabla \bfv) \bfB - \bfB (\nabla \bfv)^{\top} \right) + \alpha (\bb - \mathbf{I}) = 0 & \textrm{in } \;\ \Omega \times (0,T), 
\label{s3}
\\\label{s4}
& \bfv = 0  & \textrm{on } \;\ \partial \Omega \times (0,T),
\end{align}
where $\bD(\bfv)=\frac12(\nabla \bfv + \nabla \bfv^{\top})$, 
and 
the forcing term and problem parameters are defined as
\begin{align*}
&    \boldsymbol f =  
    \begin{cases}
        \boldsymbol{f}_F & \textrm{in } \Omega_F(t),
        \\
        \boldsymbol{f}_S & \textrm{in } \Omega_S(t),
    \end{cases}
   \qquad
    \rho := \begin{cases}
        \rho_F & \textrm{in } \Omega_F(t),
        \\
        \rho_S & \textrm{in } \Omega_S(t),
    \end{cases}
    \qquad
    \mu := \begin{cases}
        \mu_F & \textrm{in } \Omega_F(t),
        \\
        \mu_S & \textrm{in } \Omega_S(t),
    \end{cases}
\\
&    G := \begin{cases}
        G_F & \textrm{in } \Omega_F(t),
        \\
        G_S & \textrm{in } \Omega_S(t),
    \end{cases}
    \qquad
    \alpha = \begin{cases}
        \alpha_F & \textrm{in } \Omega_F(t)
        \\
        0 & \textrm{in } \Omega_S(t),
    \end{cases}
\end{align*}
with $\rho_F$ and $\rho_S$ denoting the fluid and structure densities, and $\mu_F$ and $\mu_S$ denoting the fluid and structure viscosities, respectively.
Note that $\mu_S = 0$ for purely elastic structures. Additionally, we introduce a small positive parameter $G_F \ll 1$ in the definition of $G$ to ensure that $G$ remains strictly positive throughout the fluid region. This regularization is essential for the convergence analysis of the modelling error, as it prevents degeneracy in the elastic stress term. Physically, this corresponds to introducing a small artificial elasticity in the fluid, whose effect can be made negligible by choosing $G_F$ sufficiently small.
The parameter $ \alpha_F$ is positive, and does not come from the physical properties of the problem. 
The ratio of 
$G_S / \alpha_F$ 
can be seen as the interface relaxation time, which needs to scale with the final time $T$.   
In the sharp interface model, we assume that we always know the exact location of the fluid and structure domains, and therefore the interface between them.

Equation~\eqref{s3} contains the upper-convected Maxwell time derivative which rotates and stretches with the deformation~\cite{Mao2023}.
Initially, we assume that the fluid is at rest, and that the strain in the undeformed configuration satisfies $\bb = \mathbf{I}$, where $\mathbf{I}$ the identity matrix. In the fluid, the elastic stress vanishes since there is no strain, whence $\bb = \mathbf{I}$ in $\Omega_F (t)$.

\subsubsection{Weak formulation and energy estimates for the sharp interface model}

We note that $\bb$ is a symmetric and positive definite tensor. The set $\mathbb{S}_{++}$
of such tensors does not 
form
a vector space.
However, by the symmetry and positive-definiteness, we have that
$\tr \bb \geq 0$, and $\tr \bb \equiv 0$ if and only if $\bb \equiv \mathbf 0$.
Thus (as done in~\cite{Lozinski2003}) we can define the nonnegative functional
\begin{equation}\label{eq:trace-functional}
    \| \bb \|_{*, \Omega} := \int_{\Omega} \tr \bb
\end{equation}
with domain 
$M_{tr}:=\{\bfA:\Omega\to\mathbb{S}_{++}~\vert~\tr\bfA\in L^1(\Omega)\}$.

Multiplying~\eqref{s1} by a test function $\bfz \in H^1_0(\Omega)$,~\eqref{s2} by $q \in L^2_0(\Omega)$, and~\eqref{s3} by $\bfA\in L^2(\Omega; \mathbb{R}^{d\times d}_{\rm sym})$, integrating by parts and adding equations together, we obtain the following weak formulation:
seek $\bfv \in L^{\infty}(0,T;L^2(\Omega))\cap L^2(0,T;H^1_0(\Omega))\cap L^2(0,T;H^{-1}(\Omega)), p \in L^2(0,T;L^2_0(\Omega))$, and $\bb \in L^2(0,T;M_{tr})$, $\|\bb\|_*\in L^{\infty}(0,T)$ and $D_t \bfB\in L^2(0, T; L^2(\Omega; \mathbb{R}^{d\times d}_{\sym}))$ such that
for a.e.~$t\in (0, T)$,
\begin{subequations}\label{eq:sharp-weak-form}
    \begin{alignat}{2}
 &       \int_\Omega 
        \partial_t(\rho\bfv)\cdot\bfz
        +  \int_\Omega  [(\bfv\cdot\nabla)\cdot (\rho \bfv)]\cdot\bfz
        + \int_\Omega 2\mu\bD(\bfv):\bD(\bfz)
        + \int_{\Omega} G \bfB:\bD(\bfz)
        - \int_\Omega p\nabla\cdot\bfz
        + \int_\Omega q\nabla\cdot\bfv
        \notag
        \\
&    \qquad    = \int_\Omega \rho\bff\cdot\bfz,
        \label{sw1}
      \\
 &      \int_{\Omega} q \nabla \cdot \bfv =0 
        \label{sw2}
        \\
&        \int_{\Omega} G D_t\bfB:\bfA 
        - \int_{\Omega} G(\nabla\bfv)\bfB:\bfA
        - \int_{\Omega} G\bfB(\nabla\bfv)^{\top}:\bfA
        + \int_{\Omega} \alpha (\bb - \mathbf{I}) :\bfA
        = 0,
        \label{sw3}
    \end{alignat}
\end{subequations}
for all $\bfz \in H^1_0(\Omega), q \in L^2_0(\Omega)$, and $\bfA \in L^2(\Omega; \mathbb{R}^{d\times d}_{\rm sym})$. 

The sharp interface model satisfies the following energy estimate. 
\begin{theorem}
    Let $(\bfv, p, \bb)$ be a sufficiently smooth solution of the sharp interface problem~\eqref{s1}--\eqref{s4} with initial conditions $\bfv(\cdot, 0) = \bfv_0$ and $\bb(\cdot, 0) = \bb_0$.
    Then, the following energy estimate holds:
    \begin{align*}
        & 
        \frac{\D}{\dt} \left(
            \frac12 \int_\Omega \rho |\bfv|^2
            + \frac{1}{2} \int_\Omega G \tr \bb
        \right)
        + \int_\Omega 2\mu |\bD(\bfv)|^2
         +\frac12 \|\alpha_F  \bfB \|_{*,\Omega_F(t)} 
    = \frac12 \alpha_F d |\Omega_F(t)|+\int_\Omega \rho \bff \cdot \bfv.
    \end{align*}
\end{theorem}
\begin{proof}
Let $\Gamma(t) := \partial\Omega_F(t)\cap\partial\Omega_S(t)$ denote the sharp interface and  $\bn_F$ denote the outward unit normal to $\partial\Omega_F(t)$.
We take $\bfz = \bfv$ in~\eqref{sw1} and $q=p$ in~\eqref{sw2}. Adding the equations together, we obtain
\begin{align}
 \frac12 \int_\Omega\rho\partial_t|\bfv|^2 
 + \int_\Omega  [(\bfv\cdot\nabla)\cdot (\rho \bfv)]\cdot\bfv
    + \int_\Omega 2 \mu |\bD(\bfv)|^2
    = \int_\Omega \rho \bff \cdot \bfv.
    \label{seV}
\end{align}
Using the Reynolds transport theorem,
we have
\begin{align*}
\frac12    \int_\Omega\rho\partial_t|\bfv|^2 
    &
    = \frac12 \frac{\D}{\dt}  \int_\Omega \rho |\bfv|^2
     - \frac{\rho_F - \rho_S}{2}\int_{\Gamma(t)}|\bfv|^2\bfv\cdot\bn_F,
\end{align*}
and the convective term simplifies as
\[
    \int_\Omega  [(\bfv\cdot\nabla)\cdot (\rho \bfv)]\cdot\bfv =  \frac{\rho_F - \rho_S}{2}\int_{\Gamma(t)}|\bfv|^2\bfv\cdot\bn_F.
\]

Next, we take the test function $\bfA = \mathbf{I}$ 
in~\eqref{sw3}. Using that
$
    \tr((\nabla\bfv)\bfB) = \tr(\bfB(\nabla\bfv)^\top) = \bfB:\nabla\bfv,
$
we obtain
\begin{align}\label{seB}
    \frac{\D}{\dt} \| G_S \bfB \|_{*,\Omega_S(t)} - 2 \int_{\Omega_S(t)} G_S \nabla \bfv : \bfB + \|\alpha_F  \bfB \|_{*,\Omega_F(t)} - \alpha_F d |\Omega_F(t)| = 0.
\end{align}
Multiplying~\eqref{seB} by $\frac12$, adding it to~\eqref{seV}, and noting that $\bb:\bD(\bfv) = \bb: \nabla \bfv$, we obtain the claimed energy estimate.
\end{proof}

\subsection{The diffuse interface model for FSI}

In this section, we present a diffuse interface model for the Eulerian FSI problem. We denote the diffuse equivalents of the variables previously introduced by a superscript ${\eps}$, where $\eps$ defines a length scale over which the interface is smeared out. 
Similarly to~\cite{Mokbel2018}, we define the diffuse interface FSI model as follows:
\begin{align}
     &   D^{\eps}_t \left( \rho(\phi) \bfv^{\eps} \right)   
        + \frac{\rho_F - \rho_S}{2} \left( M(\phi) \nabla m \cdot \nabla \right) \bfv^{\eps} 
        - \nabla \cdot ( 2 \mu(\phi) \bD(\bfv^{\eps}) + G(\phi) \bfB^{\eps} )
        \notag
        \\ 
        &
        \qquad
               + \nabla p^{\eps}
        = - \gamma \epsilon \nabla \cdot \left( \nabla \phi \otimes \nabla \phi \right)
        + \rho(\phi)\boldsymbol f 
        &
        \textrm{in} \; \Omega \times (0,T),
        \label{m1}
        \\
     &   \nabla \cdot \bfv^{\eps} = 0 
        &
        \textrm{in} \; \Omega \times (0,T),
        \label{m2}
                \\
     &   G(\phi) \left( D^{\eps}_t \boldsymbol B^{\eps} - (\nabla \bfv^{\eps}) \bfB^{\eps} - \bfB^{\eps} (\nabla \bfv^{\eps})^{\top} \right) 
        + \alpha(\phi) (\bfB^{\eps} - \mathbf{I})
        = 0
        &
        \textrm{in} \; \Omega \times (0,T),
        \label{m5}
        \\
 &       D^{\eps}_t \phi 
        = \nabla \cdot \left( M(\phi) \nabla m \right)
        &
        \textrm{in} \; \Omega \times (0,T),
        \label{m3}
        \\
    &    \frac12 G'(\phi)  \tr \bfB^{\eps} 
        + \gamma \left( \frac{1}{\epsilon} W'(\phi) - \eps \Delta \phi \right) 
        = m
        &
        \textrm{in} \; \Omega \times (0,T),
        \label{m4}
        \\
     &   \bfv^{\eps} = 0 
        &
        \textrm{on} \; \partial \Omega \times (0,T),
        \label{m6}
        \\
    &    \nabla \phi \cdot \boldsymbol n = 0
        &
        \textrm{on} \; \partial \Omega \times (0,T),
        \label{m7}
        \\
   &     \nabla m \cdot \boldsymbol n = 0
        &
        \textrm{on} \; \partial \Omega \times (0,T).
        \label{m8}
\end{align}
where $D_t^{\eps} = \partial_t + \bfv^{\eps} \cdot \nabla$ denotes the material derivative with respect to the diffuse velocity,
$\phi$ is the phase-field function, $M: [-1,1] \rightarrow [0, \infty)$ is a mobility function which might depend on 
the value of 
$\phi$, $\boldsymbol n$ is the outward-pointing unit normal to $\partial \Omega$, and $\gamma$ is the scaled surface tension, which is related to the physical surface tension $\tilde{\gamma}$ by $\gamma = \frac{3}{2 \sqrt{2}} \tilde{\gamma}$. The term $-\gamma \epsilon \nabla \cdot \left( \nabla \phi \otimes \nabla \phi \right)$ in~\eqref{m1} accounts for the capillary forces due to surface tension. The function $W$ is a free energy density. We set
\begin{align*}
    &\rho(\phi) := \rho_F (1+\phi)/2 + \rho_S (1-\phi)/2,
    \\
    &\mu(\phi) := \mu_F (1+\phi)/2 + \mu_S (1-\phi)/2,
    \\
    & G(\phi) := {G_F (1+\phi)/2} + G_S (1-\phi)/2,
    \\
    & \alpha(\phi) := \alpha_F (1 + \phi)/2.
\end{align*}
Initially, we assume that $\boldsymbol u = 0$ in $\Omega$, and we set $\phi=1$ in $\Omega_F$ and $\phi=-1$ in $\Omega_S$. The evolution of $\phi$ is then governed by~\eqref{m3}--\eqref{m4}.  
The diffuse interface model presented in this section is similar to the one used in~\cite{Mokbel2018}, which is based on a thermodynamically consistent derivation, the 
principal 
difference being that in our work a different constitutive model for the solid is used.

\begin{assumption}\label{assump:W}
We assume that $W\in C([-1,1]) \cap C^2((-1,1))$ and that $W'$ satisfies
\[
    \lim_{s \rightarrow -1} W'(s) = -\infty, \quad
    \lim_{s \rightarrow 1} W'(s) = \infty, \quad
    W''(s) \geq -\beta,
\]
for some 
$\beta \in \mathbb{R}$. For $x \notin [-1,1],$ we extend $W(x)$ by $+\infty$. Hence, $\int_{\Omega} W(\phi) < \infty$ implies $\phi(\bfx) \in [-1,1]$ for almost every $\bfx \in \Omega$.
\end{assumption}

\begin{remark}
As noted in~\cite{Abels2009}, Assumption~\ref{assump:W} is motivated by the free energy 
originally
suggested by Cahn and Hilliard~\cite{Cahn1958}:
\[
    W(\phi) = \frac{\theta}{2} \left((1+\phi) \ln(1+\phi) + (1-\phi)\ln(1-\phi)   \right) - \frac{\theta_{\phi}}{2} \phi^2,
\]
where $0<\theta <\theta_{\phi}$. It was shown in~\cite{Elliott1991} that using this free energy density, the associated Cahn--Hilliard equation has  a unique solution satisfying $\phi \in (-1,1)$ almost everywhere at all times. 
However, in the theory of the Cahn--Hilliard equations, this free energy is usually approximated by a suitable smooth free energy density. But due to the lack of a comparison principle for fourth order diffusion equations, one cannot ensure that the values of $\phi$ remain in $[-1,1]$. For further references, please see~\cite{Abels2007}. 
\end{remark}

\subsubsection{Weak formulation and energy estimates for the diffuse interface model}

To obtain the weak formulation, we multiply~\eqref{m1} by $\bfz^{\eps} \in H^1_0(\Omega)$,~\eqref{m2} by $q^{\eps} \in L^2_0(\Omega)$,~\eqref{m3} by $\kappa\in H^1(\Omega)$, and~\eqref{m4} by $\psi\in H^1(\Omega)$.
Furthermore, we 
contract~\eqref{m5} with $\bfA^\eps\in L^2(\Omega; \mathbb{R}^{d\times d}_{\rm sym})$,
add equations together, and integrate by parts. The  weak formulation is given as follows: seek $\bfv^{\eps} \in L^{\infty}(0,T;L^2(\Omega))\cap L^2(0,T;H^1_0(\Omega))\cap L^2(0,T;H^{-1}(\Omega))$, $p^{\eps}\in L^2(0,T;L^2_0(\Omega))$, and $\bfB^{\eps} \in L^2(0,T;M_{tr})$, $\|\bfB^{\eps}\|_*\in L^{\infty}(0,T)$ and $D^{\eps}_t \bfB^{\eps}\in L^2(0, T; L^2(\Omega; \mathbb{R}^{d\times d}_{\sym}))$ such that
for a.e.~$t\in (0, T)$,
\begin{equation}
    \begin{aligned}\label{eq:diffuse-weak-form}
        &\int_\Omega D_t^{\eps} (\rho(\phi)\bfv^\eps)\cdot\bfz^{\eps} 
        + \frac{\rho_F - \rho_S}{2} \int_\Omega \left( M(\phi) \nabla m \cdot \nabla \right) \bfv^{\eps}  \cdot \bfz^{\eps}
        + \int_\Omega 2\mu(\phi)\bD(\bfv^\eps):\bD(\bfz^{\eps})
        \\&
                + \int_\Omega G(\phi)\bfB^\eps: \nabla \bfz^{\eps}
        - \int_\Omega p^\eps \nabla \cdot \bfz^{\eps}
        + \int_\Omega q^\eps \nabla \cdot \bfv^\eps
        + \int_\Omega G(\phi)\partial_t \bfB^\eps: \bfA^{\eps}
                        + \int_{\Omega} G(\phi) \left((\bfv^{\eps} \cdot \nabla) \bfB^{\eps} \right): \bfA^{\eps}
        \\&
        - \int_\Omega G(\phi) (\nabla \bfv^{\eps})\bfB^{\eps}: \bfA^{\eps}
        - \int_\Omega G(\phi) \bfB^\eps(\nabla \bfv^{\eps})^\top: \bfA^{\eps}
        + \int_\Omega \alpha(\phi) ( \bb^{\eps} - \mathbf{I}): \bfA^{\eps}
                \\&
        = \gamma \epsilon \int_{\Omega} \nabla \phi \otimes \nabla \phi:\nabla\bfz^{\eps} 
        + \int_{\Omega} \rho(\phi)\boldsymbol f \cdot \bfz^{\eps}, 
    \end{aligned}
\end{equation}
for all $\bfz^\eps\in H^1_0(\Omega)^d, q^\eps\in L^2_0(\Omega)$, and $\bfA^\eps\in L^2(\Omega;\mathbb{R}^{d\times d}_{\rm sym})$.
Similarly, the Cahn--Hilliard problem in the weak form reads as:
seek $\phi\in L^{\infty}(0,T;H^1(\Omega))\cap H^1(0,T;(H^1(\Omega))^*)$, $m\in L^2(0,T;H^1(\Omega))$ such that
for a.e.~$t\in (0, T)$,
    \begin{align}
        \int_\Omega (D^\eps_t \phi) \kappa
        + \int_\Omega M(\phi)\nabla m \cdot\nabla \kappa
   +     \frac12\int_\Omega G'(\phi) \tr \bfB^{\eps} \psi
        + \frac{\gamma}{\eps}\int_\Omega W'(\phi) \psi
        + \gamma \eps\int_\Omega\nabla\phi\cdot\nabla \psi
        \notag
        \\
        = \int_\Omega m \psi,
       \label{eq:ch-weak-form}
    \end{align}
for all $\kappa, \psi\in L^2(0, T; H^1(\Omega))$.
Analogously to the sharp interface case, we can derive formal energy estimates for the diffuse interface model as follows.
\begin{theorem}
Let $(\bfv^{\eps}, p^{\eps}, \phi, m, \bfB^{\eps})$ be a sufficiently smooth solution of~\eqref{m1}--\eqref{m8}. Then, the following energy estimate holds:
\begin{align}
  \frac{\D}{\dt} \int_{\Omega} \frac{\rho(\phi)}{2} |\bfv^{\eps} |^2
  + \frac12 \frac{\D}{\dt} \|  G(\phi)  \tr \bfB^{\eps}  \|_{*,\Omega}
                +\frac{\gamma}{\eps} \frac{\D}{\dt} \int_\Omega W(\phi) 
                                  +  \frac{\gamma \eps}{2} \frac{\D}{\dt} \int_\Omega |\nabla\phi|^2
                                          + \int_\Omega 2\mu(\phi)|\bD(\bfv^{\eps}  )|^2
    \notag
    \\
            + \int_\Omega M(\phi) |\nabla m |^2
            +\frac{1}{2} \| \alpha(\phi)  \tr \bfB^{\eps} \|_{*,\Omega} 
    =     \frac{\alpha_F d}{4}  |\Omega| 
 +    \frac{\alpha_F d}{4}   \int_{\Omega} \phi_0  +  \int_{\Omega} \rho(\phi) \boldsymbol f.
\notag
\end{align}
\end{theorem}
\begin{proof}
To obtain the energy estimates, we take $\bfz^{\eps} = \bfv^{\eps}, q^{\eps} = p^{\eps}, \kappa = m, \psi= D_t^{\eps} \phi$, and $\bfA^\eps = \frac12\mathbf{I}$, 
and add together~\eqref{eq:diffuse-weak-form} and~\eqref{eq:ch-weak-form},
obtaining
\begin{align*}
    \int_\Omega \partial_t (\rho(\phi)\bfv^{\eps} )\cdot \bfv^{\eps}
    + \int_\Omega \left(\bfv^{\eps} \cdot \nabla ( \rho(\phi)\bfv^{\eps} ) \right) \cdot \bfv^{\eps}  
    + \frac{\rho_F - \rho_S}{2} \int_\Omega \left( M(\phi) \nabla m \cdot \nabla \right) \bfv^{\eps}  \cdot \bfv^{\eps} 
    + \int_\Omega 2\mu(\phi)|\bD(\bfv^{\eps} )|^2
    \notag
    \\
    + \int_{\Omega} G(\phi) \bfB^{\eps} : \nabla \bfv^{\eps}
    + \int_\Omega M(\phi) |\nabla m |^2
    + \frac12 \int_\Omega G'(\phi) \tr \bfB^{\eps} D^{\eps}_t \phi 
    + \frac{\gamma}{\eps} \int_\Omega D^{\eps} _t W(\phi) 
    + \gamma \eps \int_\Omega \nabla\phi \cdot \nabla D^{\eps}_t  \phi
    \notag
    \\
    + \frac12 \int_\Omega G(\phi)D^{\eps}_t (\tr\bfB^{\eps}  ) 
    - \int_{\Omega} G(\phi) \nabla \bfv^{\eps} : \bfB^{\eps}
    + \frac12\int_\Omega \alpha(\phi) ( \tr \bb^{\eps} - d) 
    \notag
    \\
    = \gamma \eps \int_\Omega \nabla \phi \otimes \nabla \phi:\nabla \bfv^{\eps}
    + \int_{\Omega} \rho(\phi) \boldsymbol f\cdot \bfv^{\eps}.
\end{align*}

Using
\begin{align*}
    \int_{\Omega} \partial_t \left( \rho(\phi) \bfv^{\eps}  \right)   \cdot \bfv^{\eps}  = 
    \int_{\Omega} \partial_t \left( \rho(\phi)\right) |\bfv^{\eps} |^2 + \int_{\Omega} \partial_t \bfv^{\eps} \cdot \rho(\phi) \bfv^{\eps}  
    =
    \int_{\Omega} \partial_t \left( \frac{\rho(\phi)}{2} |\bfv^{\eps} |^2 \right)+ \int_{\Omega} \partial_t \left(\rho(\phi)\right) \frac{|\bfv^{\eps} |^2}{2},
\end{align*}
and noting that
\begin{align*}
\partial_t \left( \rho(\phi) \right) = \rho'(\phi) \partial_t \phi = \frac{\rho_F - \rho_S}{2}\partial_t\phi,
\end{align*}
by employing~\eqref{m3}, we can write 
\begin{align*}
       \int_\Omega \partial_t (\rho(\phi)\bfv^{\eps}  )\cdot \bfv^{\eps}  =
            \frac{\D}{\dt} \int_{\Omega}
        \frac{\rho(\phi)}{2} |\bfv^{\eps} |^2
        \underbrace{+ \frac{\rho_F - \rho_S}{4} \int_{\Omega}  \nabla \cdot \left( M(\phi) \nabla m \right) |\bfv^{\eps} |^2}_{\mathcal{T}_1}
        -\underbrace{\frac{\rho_F - \rho_S}{4} \int_{\Omega}  (\bfv^{\eps} \cdot \  \nabla \phi) |\bfv^{\eps} |^2}_{\mathcal{T}_2}.
\end{align*}
For $\mathcal{T}_1$, after integration by parts, we have
\begin{align*}
\mathcal{T}_1
 = -  \frac{\rho_F - \rho_S}{2} \int_{\Omega}  \left( \left( M(\phi) \nabla m  \cdot \nabla\right)  \bfv^{\eps}  \right) \cdot \bfv^{\eps} .
\end{align*}
For integral $\mathcal{T}_2$, using $\displaystyle\frac{\rho_F-\rho_S}{2} \nabla \phi = \nabla \rho(\phi)$ and integrating by parts, we obtain
\begin{align*}
\mathcal{T}_2 &
=
-\frac{ 1}{2}  \int_{\Omega} \bfv^{\eps}  \cdot \nabla \rho(\phi)  |\bfv^{\eps} |^2 
\notag
\\ 
&
=\frac{ 1}{2} \int_{\Omega} \nabla \cdot \bfv^{\eps}   \rho(\phi)   |\bfv^{\eps} |^2
-\frac{ 1}{2} \int_{\Omega} \nabla \cdot (\bfv^{\eps}   \rho(\phi) )  |\bfv^{\eps} |^2
\notag
\\
&
=
\frac{ 1}{2} \int_{\Omega} \rho(\phi) \bfv^{\eps}  \cdot \nabla|\bfv^{\eps} |^2
-\frac{ 1}{2} \int_{\partial \Omega}  \rho(\phi)   |\bfv^{\eps} |^2 \bfv^{\eps}   \cdot \boldsymbol{n}
\notag
\\
&
=
 \int_{\Omega} \rho(\phi)\left(  (\bfv^{\eps}  \cdot \nabla) \bfv^{\eps}  \right) \cdot \bfv^{\eps} . 
\end{align*} 
For the convective term, we have
\begin{align}
\int_{\Omega} (\bfv^{\eps}  \cdot \nabla) (\rho(\phi)\bfv^{\eps} ) \cdot \bfv^{\eps}  &
=
\sum_{i,j}  \int_\Omega v^{\eps}_j  \frac{\partial (\rho(\phi) v^{\eps}_i )}{\partial x_j}  v^{\eps}_i
\notag
\\
&
=
-  \sum_{i,j}  \int_\Omega \frac{\partial ( v^{\eps}_j  v^{\eps}_i)}{\partial x_j}  \rho(\phi) v^{\eps}_i 
\notag
\\
&
=
-  \int_\Omega \nabla \cdot \bfv^{\eps}  \rho(\phi) |\bfv^{\eps} |^2
-    \int_\Omega  \rho(\phi) (\bfv^{\eps}  \cdot \nabla )  \bfv^{\eps}   \cdot \bfv^{\eps} 
\notag
\\
&
=
-    \int_\Omega  \rho(\phi) \left( (\bfv^{\eps}  \cdot \nabla )  \bfv^{\eps}  \right)  \cdot \bfv^{\eps} .
\label{convective}
\end{align}
Using integration by parts and $\nabla \cdot \bfv^{\eps} = 0$, we obtain the following equality:
\begin{align*}
\frac{\gamma}{\eps}  \int_\Omega D^{\eps}_t W(\phi) 
&=
\frac{\gamma}{\eps}  \int_\Omega \partial_t W(\phi)
+\frac{\gamma}{\eps}  \int_\Omega {\bfv^{\eps} } \cdot \nabla (W(\phi))
\\
&=
\frac{\gamma}{\eps} \frac{\D}{\dt} \int_\Omega  W(\phi)
-\frac{\gamma}{\eps}  \int_\Omega \nabla \cdot {\bfv^{\eps} } W(\phi)
\\
&=
\frac{\gamma}{\eps} \frac{\D}{\dt} \int_\Omega  W(\phi).
\end{align*}
Furthermore, by integration by parts, we have
\begin{align*}
\gamma \eps \int_\Omega \nabla\phi \cdot \nabla D^{\eps}_t  \phi
& =
\gamma \eps \int_\Omega \nabla\phi \cdot \nabla \partial_t  \phi
+
\gamma \eps \int_\Omega \nabla\phi \cdot \nabla \left( \bfv^{\eps}  \cdot \nabla\phi \right)
\\
& =
\frac{\gamma \eps}{2} \frac{\D}{\dt} \int_\Omega |\nabla\phi|^2
-
\gamma \eps \int_\Omega \Delta \phi  \left(\nabla\phi \cdot  \bfv^{\eps}    \right)
\\
& =
\frac{\gamma \eps}{2} \frac{\D}{\dt} \int_\Omega |\nabla\phi|^2
-
\gamma \eps \int_\Omega \nabla \cdot   \left(\nabla\phi \otimes \nabla \phi \right) \cdot \bfv^{\eps}  
+
\gamma \eps \int_\Omega \nabla   \left(\frac12 |\nabla\phi |^2 \right) \cdot \bfv^{\eps}  
\\
& =
\frac{\gamma \eps}{2} \frac{\D}{\dt} \int_\Omega |\nabla\phi|^2
-
\gamma \eps \int_\Omega \nabla \cdot   \left(\nabla\phi \otimes \nabla \phi \right) \cdot \bfv^{\eps}  
-
\frac{\gamma \eps}{2} \int_\Omega   |\nabla\phi |^2  \nabla \cdot \bfv^{\eps}  
\\
& =
\frac{\gamma \eps}{2} \frac{\D}{\dt} \int_\Omega |\nabla\phi|^2
-
\gamma \eps \int_\Omega \nabla \cdot   \left(\nabla\phi \otimes \nabla \phi \right) \cdot \bfv^{\eps} .
\end{align*}
To handle the terms arising from the upper convected derivative of $\bfB^{\eps}$, we note that the following holds:
\begin{align*}
  \int_{\Omega}  G(\phi)  \partial_t \tr \bfB^{\eps} 
  &
  =  \frac{\D}{\dt}\int_{\Omega}  G(\phi)  \tr \bfB^{\eps}  
  - \int_{\Omega}  \partial_t \left(G(\phi) \right)  \tr \bfB^{\eps}  
  \\
    &
  =  \frac{\D}{\dt}\int_{\Omega}  G(\phi)  \tr \bfB^{\eps}  
  - \int_{\Omega}  G'(\phi) \partial_t \phi  \tr \bfB^{\eps},
\end{align*}
and
\begin{align*}
\int_{\Omega} G(\phi) (\bfv^{\eps}  \cdot \nabla) \tr \bfB^{\eps}   
&= -  \int_{\Omega} \nabla G(\phi) \cdot  \bfv^{\eps}  (\tr \bfB^{\eps} ) - \int_{\Omega} G(\phi)  \nabla \cdot \bfv^{\eps}   \tr \bfB^{\eps} 
 \notag
 \\
&= -  \int_{\Omega} G'(\phi)\nabla \phi \cdot  \bfv^{\eps}  (\tr \bfB^{\eps} ),
\end{align*}
so that 
\begin{align*}
 \frac12 \int_\Omega G(\phi)D^{\eps}_t (\tr\bfB^{\eps}  ) 
 = \frac12 \frac{\D}{\dt}\int_{\Omega}  G(\phi)  \tr \bfB^{\eps}  
-\frac12 \int_\Omega G'(\phi)D^{\eps}_t \phi \tr\bfB^{\eps}.
\end{align*}
Integrating~\eqref{m3} over $\Omega$, integrating by parts the convective term, and using the divergence theorem, we have
\begin{align*}
    0
    &= \int_{\Omega} \partial_t \phi
    + \int_{\Omega} \bfv^{\eps} \cdot \nabla \phi
    + \int_{\partial \Omega}  M(\phi) \nabla m \cdot \bn
    = \frac{\D}{\dt}\int_\Omega\phi.
\end{align*}
Therefore, using that $\int_\Omega\phi = \int_\Omega\phi_0$ for  all $t$, we have
\begin{align*}
\frac12 \int_{\Omega} \alpha(\phi)  \tr (\bfB^{\eps} - \mathbf{I})
&=  \frac{1}{2} \| \alpha(\phi)  \tr \bfB^{\eps} \|_{*,\Omega}
- \frac{\alpha_F d}{4} \int_{\Omega}(1+\phi)  
\\
&=  \frac{1}{2} \| \alpha(\phi)  \tr \bfB^{\eps} \|_{*,\Omega}
- \frac{\alpha_F d}{4}  |\Omega| 
- \frac{\alpha_F d}{4}   \int_{\Omega} \phi_0.
\end{align*}
Collecting the estimates above and using $\nabla \bfv^{\eps} : \bfB^{\eps} + \bfB^{\eps}: (\nabla \bfv^{\eps})^{\top} = 2 \bfB^{\eps} :  \nabla \bfv^{\eps},$ we obtain the  claimed energy estimate.
\end{proof}

\section{Modeling error analysis}\label{modelingerror}

We start this section with stating the additional assumptions needed for the modeling error analysis. We assume  $\rho_F = \rho_S=\rho$. We also define the following quantities that we assume are bounded independently of $\eps$:
\begin{align}\nonumber
	\Lambda(t):=& {  2 C_K  \|\bD(\bfv^{\eps})\|_{L^\infty(\Omega)^{d\times d}}} +\frac{C_K^2}{2\delta_1}+\frac{8\delta_2}{G_*}+\frac{2\delta_3}{\rho}
	+\frac{1}{G_*}\|D_t^{\eps} G(\phi)\|_{L^\infty(\Omega)}+\frac{G_S}{\delta_3}\|\nabla\bfB \|^2_{L^{\infty}(\Omega)^{{d\times d\times d}}}
	\\&
	+\frac{1}{G_*}\Big (2\delta_2\|\nabla\bfv\|^2_{L^{\infty}(\Omega)^{{d\times d}}}
	+ 4 \|G(\phi)\|^2_{L^{\infty}{(\Omega)}}  \| \nabla \bfv^{\eps} \|^2_{L^{\infty}(\Omega)^{{d\times d}}} 
	\notag
	\\
	&+\frac{2C_K^2}{ \delta_1}\|G(\phi)\|^2_{L^{\infty}{(\Omega)}} \|\bb \|^2_{L^{\infty}(\Omega)^{{d\times d}}} \Big ),
	\label{GronwalLambda}
\end{align}
\begin{align}\label{PFConstant}
 F=\max_{t\in [0,T]}\left (\frac{\alpha_F^2}{16\delta_2}+\frac{1}{\delta_1} \|\bD(\bfv(t)) \|^2_{L^{\infty}(\Omega)^{{d\times d}}}+ \left(\frac{C_K^2}{4\delta_1}+\frac{1}{\delta_2}+\frac{\alpha_F^2}{16\delta_2}\right)\|\bb(t)\|^2_{L^\infty(\Omega)^{{d\times d}}}\right .
 \\\nonumber
 \left .
 +\frac{1}{4\delta_2}\| \partial_t  \bfB(t) \|^2_{L^{\infty}(\Omega)^{{d\times d}}} 
 +\frac{1}{4 \delta_2}\|\nabla\bfB \|^2_{L^{\infty}(\Omega)^{{d\times d\times d}}}\|\bfv\|^2_{L^{\infty}(\Omega)^d}\right ).
\end{align}
The function $\Lambda(t)$ will play the role of a Gr\"onwal coefficient in the modeling error estimate, while $F$ is a constant multiplying the phase-field error.
We  introduce the following notation for the minimum values of the material parameters:
$$
\mu_* := \min\{\mu_F,\mu_S\},\quad G_*=\min\{G_F,G_S\},
$$
and define the energy of the modeling error:
\begin{equation}\label{eq:def-modelling-error-energy}
\cE(t)
:= \frac{\rho}{2}\|\bfv(t)-\bfv^\eps(t)\|_{L^2(\Omega)^d}^2
+ \frac{G_*}{2}\int_\Omega |\bfB(t)-\bfB^\eps(t)|^2.
\end{equation}
Now we are in a position to state the main theorem of this section.

\begin{theorem}\label{thm:modelling-error}
Assume that $(\bfv, p, \bb)$ is a sufficiently regular solution of the sharp interface problem~\eqref{s1}--\eqref{s4} and $(\bfv^\eps, p^\eps, \bb^\eps, \phi)$ is a sufficiently regular solution of the diffuse interface problem~\eqref{m1}--\eqref{m8}, both defined on the time interval $[0,T]$.
Assume further that $\Lambda \in L^1(0,T)$, where $\Lambda$ is defined in~\eqref{GronwalLambda}, and that the constant $F$ defined in~\eqref{PFConstant} is finite.
Then, for every $t\in [0,T]$, the following modeling error estimate holds:
\begin{align}\notag
    \cE(t)&+
  \int_0^t\int_\Omega 2\mu(\phi)|\bD(\bfv - \bfv^\eps)|^2
  + \int_0^t\int_\Omega \alpha(\phi)|\bfB-\bfB^{\eps}|^2 
  \\ \label{modelling-error-estimate}
  &\leq \exp\left (\int_0^t \Lambda(s) ds\right )\cE(0)
  + F\int_0^t \exp\left (\int_s^t \Lambda(\tau) d\tau\right )\|\phi(s)-(\mathbbm{1}_F - \mathbbm{1}_S)(s)\|^2_{L^2(\Omega)} ds
  \\\notag
  &+ \gamma^2\eps^2\frac{C_K^2}{\delta}\int_0^t \exp\left (\int_s^t \Lambda(\tau) d\tau\right )\|\nabla\phi(s)\|^4_{L^4(\Omega)} ds.
\end{align}
\end{theorem}

\begin{remark}
    Theorem~\ref{thm:modelling-error} shows that the modeling error measured in the energy norm defined in~\eqref{eq:def-modelling-error-energy} can be controlled by the initial modeling error, the phase-field approximation error, and a term of the order of $\eps^2$ depending on the gradient of the phase-field. In particular, if the phase-field $\phi$ converges to the characteristic function of the fluid domain in $L^\infty(0,T;L^2(\Omega))$ as $\eps\to 0$, then the modeling error converges to zero as $\eps\to 0$.
\end{remark}

\begin{remark}
    The assumption that $\rho_F = \rho_S$ is made to avoid technical complications arising from the convective term in the momentum equation. We expect that with additional effort, this assumption can be removed. Moreover, notice that the addition of a small artificial elasticity in the fluid region was necessary since $G_*$ appears in the denominator of several terms in~\eqref{GronwalLambda}. We believe that this is an artefact of the analysis, which we confirm numerically in Section~\ref{sec:results}.
\end{remark}

\begin{remark}
    Theorem~\ref{thm:modelling-error} requires sufficient regularity of the sharp interface solution, which is a standard assumption in the numerical analysis of diffuse interface models. Additionally, we assume that $D^\eps_t\phi$ and $\nabla\bfv^{\eps}$ are bounded in $L^1(0,T; L^\infty(\Omega))$. These regularity and uniform boundedness requirements on the diffuse interface solution are non-standard and more difficult to justify theoretically. We verify these assumptions numerically by demonstrating that these quantities remain bounded in our computational experiments. 
\end{remark}

\begin{proof}
To obtain the modeling error equations, we subtract the diffuse interface weak form~\eqref{eq:diffuse-weak-form} from the sharp interface weak form~\eqref{eq:sharp-weak-form}, which gives rise to the following error equation:
\begin{align}
        &\int_\Omega \rho \left(   D_t(\bfv) - D_t^{\eps} \bfv^\eps \right) \cdot\bfz 
        + \int_\Omega 2(  \mu\bD(\bfv) - \mu(\phi)\bD(\bfv^\eps)):\bD(\bfz)
        + \int_\Omega ( G \bfB -  G(\phi)\bfB^\eps):\nabla \bfz
        \notag
        \\&
           - \int_\Omega (p - p^{\eps})\nabla\cdot\bfz
        + \int_\Omega q\nabla\cdot( \bfv - \bfv^{\eps} )
    + \int_\Omega (  GD_t\bb - G(\phi)D^\eps_t\bb^\eps ):\bfA
\notag
\\
      &  
          - \int_\Omega (G(\nabla\bfv)\bb - G(\phi)(\nabla\bfv^\eps)\bb^\eps):\bfA
       - \int_\Omega (G\bb(\nabla\bfv)^\top - G(\phi)\bb^\eps(\nabla\bfv^\eps)^\top):\bfA
       \notag
       \\
       & 
              + \int_\Omega (\alpha\bb- \alpha(\phi)\bb^\eps - (\alpha - \alpha(\phi))\mathbf{I}):\bfA
              = \int_\Omega( \rho -\rho(\phi))\bff\cdot\bfz     
       - \gamma\eps\int_\Omega(\nabla\phi\otimes\nabla\phi):\nabla \bfz,
       \label{eq:modelling-error-eqns}
\end{align}
for all test functions $\bfz\in H^1_0(\Omega)^d, q\in L^2_0(\Omega), \bfA\in L^2(\Omega;\mathbb{R}^{d\times d}_{\rm sym})$.

Let $\bfz = \bfv - \bfv^\eps, q = p - p^\eps$, 
and $\bfA = \bfB -\bfB^\eps $ 
in~\eqref{eq:modelling-error-eqns}.
Then the error equation becomes
\begin{align*}
\sum_{i=1}^ {10} \mathcal{I}_i = \mathcal{R }
\end{align*}
where
\begin{align*}
\mathcal{I}_1 &=   \rho\int_\Omega \partial_t (\bfv - \bfv^\eps)\cdot(\bfv - \bfv^\eps) = \frac{\rho}{2}\frac{\D}{\dt}\|\bfv - \bfv^\eps\|^2_{L^2(\Omega)^d},
\\
\mathcal{I}_2 &=   \rho\int_\Omega [(\bfv\cdot\nabla)\bfv]\cdot(\bfv - \bfv^\eps)
    - \rho\int_\Omega [(\bfv^\eps\cdot\nabla)\bfv^\eps]\cdot(\bfv - \bfv^\eps),
    \\
\mathcal{I}_3 &=\int_\Omega 2(\mu\bD(\bfv) - \mu(\phi)\bD(\bfv^\eps)):\bD(\bfv - \bfv^\eps),
    \\
\mathcal{I}_4 &= \int_\Omega (G \bfB - G(\phi)\bfB^\eps):\nabla(\bfv - \bfv^\eps),
\\
\mathcal{I}_5 &=  \int_\Omega(G \partial_t  \bfB - G(\phi)\partial_t \bfB^\eps):(\bfB-\bfB^{\eps}),  
\\
\mathcal{I}_6 &=   \int_\Omega \left(G (\bfv  \cdot \nabla) \bfB - G(\phi) (\bfv^{\eps}  \cdot \nabla) \bfB^\eps \right ):(\bfB-\bfB^{\eps}),
    \\
\mathcal{I}_7 &=
    - \int_\Omega (G(\nabla\bfv)\bb - G(\phi)(\nabla\bfv^\eps)\bb^\eps):(\bfB-\bfB^{\eps}),
    \\
\mathcal{I}_8 &=    - \int_\Omega (G\bb(\nabla\bfv)^\top - G(\phi)\bb^\eps(\nabla\bfv^\eps)^\top):(\bfB-\bfB^{\eps}),
    \\
\mathcal{I}_9 &= \int_\Omega (\alpha \bfB - \alpha(\phi) \bfB^\eps):(\bfB-\bfB^{\eps}),
\\
\mathcal{I}_{10} &=
    - \int_\Omega  (\alpha - \alpha(\phi) \mathbf{I}):(\bfB-\bfB^{\eps}),
     \\
\mathcal{R} &=  - \gamma\eps\int_\Omega(\nabla\phi\otimes\nabla\phi):\nabla(\bfv - \bfv^\eps).
\end{align*}
Using the identity $\int_\Omega [(\bfv\cdot\nabla)\bfz]\cdot\bfz = 0$ for all $\bfz, \bfv \in H^1_0(\Omega)^d$ with $\nabla\cdot\bfv = 0$,  and the Korn inequality with constant $C_K$, we have
\begin{align*}
   \mathcal{I}_2
& =\rho\int_{\Omega}\left( (\bfv-\bfv^\eps)\cdot\nabla\bfv^{\eps}+\bfv\cdot\nabla(\bfv - \bfv^\eps)\right) \cdot (\bfv - \bfv^\eps)
    \\
    &=\rho\int_{\Omega}\left( (\bfv-\bfv^\eps)\cdot\nabla\bfv^{\eps}\right) \cdot (\bfv - \bfv^\eps)
    \\
        &
   \leq { \rho C_K \|\bfv - \bfv^\eps\|^2_{L^2(\Omega )^d}\|\bD(\bfv^{\eps})\|_{L^\infty(\Omega)^{d\times d}}. }
\end{align*}
The fluid viscous term can be decomposed as follows:
\begin{align*}
   \mathcal{I}_3 &=
     \int_\Omega 2(\mu - \mu(\phi))\bD(\bfv):\bD(\bfv - \bfv^\eps)
    + 
     2\int_\Omega \mu(\phi)|\bD(\bfv - \bfv^\eps)|^2,
 \end{align*}
 where the first term can be bounded as    
\begin{align*}
      \int_\Omega 2(\mu - \mu(\phi))\bD(\bfv):\bD(\bfv - \bfv^\eps) & \leq 2  \|\bD(\bfv) \|_{L^{\infty}(\Omega)^{{d\times d}}} \|\mu - \mu(\phi)\|_{L^2(\Omega)} \|\bD(\bfv - \bfv^\eps)\|_{L^2(\Omega)^{{d\times d}}}
      \\
      & \leq 
      \frac{1}{\delta_1}  \|\bD(\bfv) \|^2_{L^{\infty}(\Omega)^{{d\times d}}} \|\mu - \mu(\phi)\|^2_{L^2(\Omega)}   +  \delta_1 \|\bD(\bfv - \bfv^\eps)\|^2_{L^2(\Omega)^{d\times d}}.
\end{align*}
Term $\mathcal{I}_4$ can be bounded as
\begin{align*}
        \mathcal{I}_4
        &=
        \int_\Omega (G  - G(\phi))\bfB:\nabla(\bfv - \bfv^\eps)
        + 
        \int_\Omega G(\phi)(\bfB - \bfB^\eps):\nabla(\bfv - \bfv^\eps)
        \\
        &\leq
        \frac{C_K^2}{4\delta_1} \|\bb\|^2_{L^\infty(\Omega)^{d \times d}} \|G -G(\phi) \|^2_{L^2(\Omega)^{d \times d}}
        +2 \delta_1 \|\bD(\bfv - \bfv^\eps)\|^2_{L^2(\Omega)^{d\times d}}
        \\&
        + \frac{C_K^2}{4\delta_1} \|G( \phi) \|^2_{L^\infty(\Omega)}\|\bb - \bb^\eps\|^2_{L^2(\Omega)^{d\times d}}
    \end{align*}
 by using the Cauchy-Schwarz, Young's inequality with $\delta_1>0$, and the Korn inequality with constant $C_K$.

Next, we have
\begin{align*}
    \mathcal{I}_5 & = 
    \int_\Omega (G - G(\phi))\partial_t  \bfB:(\bfB-\bfB^{\eps}) 
  +  \frac12  \int_\Omega G(\phi) \frac{d}{dt}  |\bfB-\bfB^{\eps}|^2
    \\
    &=
   \int_\Omega (G - G(\phi))\partial_t  \bfB:(\bfB-\bfB^{\eps}) 
  +  \frac12   \frac{\D}{\dt}  \int_\Omega G(\phi)  |\bfB-\bfB^{\eps}|^2
 -   \frac 1 2 \int_\Omega  \partial_t G(\phi)   |\bfB-\bfB^{\eps}|^2,
\end{align*}
where using the Cauchy-Schwarz and Young's inequality with $\delta_2>0$, we can bound 
\begin{align*}
  \int_\Omega (G - G(\phi))\partial_t  \bfB:(\bfB-\bfB^{\eps}) 
  &
  \leq  \| \partial_t  \bfB \|_{L^{\infty}(\Omega)^{d\times d}} \|G - G(\phi))\|_{L^2(\Omega)}  \|\bfB-\bfB^{\eps}\|_{L^2(\Omega)^{d\times d}}
  \\
    &
  \leq  \frac{1}{4 \delta_2}\| \partial_t  \bfB \|^2_{L^{\infty}(\Omega)^{d\times d}} \|G - G(\phi))\|^2_{L^2(\Omega)} + \delta_2 \|\bfB-\bfB^{\eps}\|^2_{L^2(\Omega)^{d\times d}}.
\end{align*}
Using the 
following
trilinear decomposition
\begin{equation}\label{eq:first-trilinear-decomp}
    G (\bfv  \cdot \nabla)\bb  - G(\phi) (\bfv^\eps \cdot \nabla) \bb^\eps
    = (G - G(\phi)) (\bfv \cdot \nabla) \bb
    + G(\phi)(\bfv^{\eps}  \cdot \nabla) (\bb - \bb^\eps)
    + G(\phi)((\bfv - \bfv^\eps)  \cdot \nabla) \bb,
\end{equation}
we have
\begin{align*}
\mathcal{I}_6 =&   \int_\Omega (G - G(\phi)) (\bfv \cdot \nabla) \bb :(\bfB-\bfB^{\eps})
    +  \int_\Omega  G(\phi)(\bfv^{\eps}  \cdot \nabla) (\bb - \bb^\eps) :(\bfB-\bfB^{\eps})
    \\
    &
    +  \int_\Omega  G(\phi)((\bfv - \bfv^\eps)  \cdot \nabla) \bb :(\bfB-\bfB^{\eps})
    \\
    \leq &
   \frac{1}{4 \delta_2} \|G-G(\phi)\|^2_{L^2(\Omega)}\| \nabla \bfB\|^2_{L^{\infty}(\Omega)^{{d\times d\times d}}}\| \bfv\|^2_{L^{\infty}(\Omega)^d} + \delta_2\| \bfB-\bfB^{\eps}\|^2_{L^2(\Omega)^{{d\times d}}}
\\
&
+  \int_\Omega  G(\phi)(\bfv^{\eps}  \cdot \nabla) (\bb - \bb^\eps) :(\bfB-\bfB^{\eps})
+\frac{G_S^2}{2\delta_3}  \| \nabla\bfB \|^2_{L^{\infty}(\Omega)^{{d\times d}}} \| \bfB-\bfB^{\eps}\|^2_{L^2(\Omega)^{{d\times d}}}  + \delta_3 \| \bfv- \bfv^{\eps}\|^2_{L^{2}(\Omega)^d}.
\end{align*}
We integrate the remaining term by parts, obtaining
\begin{align*}
  \int_\Omega  G(\phi)(\bfv^{\eps}  \cdot \nabla) (\bb - \bb^\eps) :(\bfB-\bfB^{\eps}) = {-\frac{1}{2}\int_{\Omega}( \nabla G(\phi) \cdot \bfv^{\eps} ) |\bfB-\bfB^{\eps}|^2. }
\end{align*}

Using the analog of trilinear decomposition~\eqref{eq:first-trilinear-decomp}
for terms in $\mathcal{I}_7$ and $\mathcal{I}_8$, we obtain
\begin{align*}
\mathcal{I}_7 + \mathcal{I}_8 & \leq
2\|G-G(\phi)\|_{L^2(\Omega)}\|\bb\|_{L^{\infty}(\Omega)^{{d\times d}}} \|\nabla \bfv\|_{L^{\infty}(\Omega)^{{d\times d}}} \|\bb - \bb^{\eps}\|_{L^2(\Omega)^{{d\times d}}}
\\
&
+2\|G(\phi)\|^2_{L^{\infty}{(\Omega)}}  \| \nabla \bfv^{\eps} \|^2_{L^{\infty}(\Omega)^{{d\times d}}}   \|\bb-\bb^{\eps}\|^2_{L^2(\Omega)^{{d\times d}}}
\\
&
+\frac{C_K^2}{\delta_1}\|G(\phi)\|^2_{L^{\infty}{(\Omega)}} \|\bb \|^2_{L^{\infty}(\Omega)} \|\bb-\bb^{\eps}\|^2_{L^2(\Omega)^{{d\times d}}}  
+ \delta_1\|\bD( \bfv-\bfv^{\eps})\|_{L^{2}(\Omega)^{{d\times d}}}
\\
&\leq \frac{1}{\delta_2}\|\bb\|^2_{L^{\infty}(\Omega)^{{d\times d}}}\|G - G(\phi) \|^2_{L^2(\Omega)} + \delta_1 \|\bD( \bfv-\bfv^{\eps})\|_{L^{2}(\Omega)^{{d\times d}}}
\\&+\Big (\delta_2\|\nabla\bfv\|^2_{L^{\infty}(\Omega)^{{d\times d}}}
+ 2 \|G(\phi)\|^2_{L^{\infty}{(\Omega)}}  \| \nabla \bfv^{\eps} \|^2_{L^{\infty}(\Omega)^{{d\times d}}} \Big)  \|\bb - \bb^{\eps}\|^2_{L^2(\Omega)^{{d\times d}}}
\\
&
 +\frac{C_K^2}{ \delta}\|G(\phi)\|^2_{L^{\infty}{(\Omega)}} \|\bb \|^2_{L^{\infty}(\Omega)^{{d\times d}}}  \|\bb - \bb^{\eps}\|^2_{L^2(\Omega)^{{d\times d}}}.
\end{align*}
Next,
we have   
\begin{align*}
\mathcal{I}_9
    &= \int_\Omega (\alpha - \alpha(\phi)) \bb:(\bfB-\bfB^{\eps})
    + \int_\Omega\alpha(\phi)|\bfB-\bfB^{\eps}|^2
\\
&
     \leq \frac{\alpha_F}{2} \| \bb \|_{L^{\infty}(\Omega)^{{d\times d}}} {\left\|\mathbbm{1}_F - \frac{1 + \phi}{2}\right\|}_{L^2(\Omega)} \|\bfB-\bfB^{\eps}\|_{L^2(\Omega)^{d\times d}}
         + \int_\Omega\alpha(\phi)|\bfB-\bfB^{\eps}|^2
         \\
&
     \leq \frac{\alpha_F^2}{16 \delta_2}  \| \bb \|^2_{L^{\infty}(\Omega)^{{d\times d}}} {\left\|\mathbbm{1}_F - \frac{1 + \phi}{2}\right\|}^2_{L^2(\Omega)} + \delta_2 \|\bfB - \bfB^{\eps}\|^2_{L^2(\Omega)^{{d\times d}}}
         + \int_\Omega\alpha(\phi)|\bfB-\bfB^{\eps}|^2.
\end{align*}
Finally,
\begin{align*}
\mathcal{I}_{10}
& \leq \frac{\alpha_F}{2}{\left\|\mathbbm{1}_F - \frac{1 + \phi}{2}\right\|}_{L^2(\Omega)} \| \bfB-\bfB^{\eps} \|_{L^2(\Omega)^{{d\times d}}}
\\
& \leq  \frac{\alpha_F^2}{16 \delta_2}{\left\|\mathbbm{1}_F - \frac{1 + \phi}{2}\right\|}^2_{L^2(\Omega)} + \delta_2 \| \bfB-\bfB^{\eps} \|^2_{L^2(\Omega)^{{d\times d}}}.
\end{align*}
The  remainder term we estimate as follows:
\begin{align*}
\mathcal{R} = - \gamma\eps\int_\Omega(\nabla\phi\otimes\nabla\phi):\nabla(\bfv - \bfv^\eps)
\leq \gamma^2\eps^2\frac{C_K^2}{4 \delta_1}\|\nabla\phi\|^4_{{L^4(\Omega)^d}}+\delta_1 \|\bD(\bfv - \bfv^\eps)\|^2_{L^2(\Omega)^{{d\times d}}}.
\end{align*}

Putting everything together, we obtain the following error estimate: 
\begin{align*}
 &\frac{\rho}{2}\frac{\D}{\dt}\|\bfv - \bfv^\eps\|^2_{L^2(\Omega)^d}
+  \frac12   \frac{\D}{\dt} \int_\Omega G(\phi) |\bfB-\bfB^{\eps}|^2+
  \int_\Omega 2\mu(\phi)|\bD(\bfv - \bfv^\eps)|^2
  + \int_\Omega \alpha(\phi)|\bfB-\bfB^{\eps}|^2
  \\
  &\leq
     \rho C_K \|\bfv - \bfv^\eps\|^2_{L^2(\Omega )^d}\|\bD(\bfv^{\eps})\|_{L^\infty(\Omega)^{d\times d}}
    +5 \delta_1 \|\bD(\bfv - \bfv^\eps)\|^2_{L^2(\Omega)^{d\times d}}
    +
     \frac{1}{\delta_1}  \|\bD(\bfv) \|^2_{L^{\infty}(\Omega)^{{d\times d}}} \|\mu - \mu(\phi)\|^2_{L^2(\Omega)}
     \\&
        +\frac{C_K^2}{4\delta_1} \|\bb\|^2_{L^\infty(\Omega)^{d\times d}} \|G -G(\phi) \|^2_{L^2(\Omega)^{d\times d}}
        + \frac{C_K^2}{4\delta_1} \|G( \phi) \|^2_{L^\infty(\Omega)}\|\bb - \bb^\eps\|^2_{L^2(\Omega)^{d\times d}}
    \\& +\frac{1}{4 \delta_2}\| \partial_t  \bfB \|^2_{L^{\infty}(\Omega)^{d\times d}} \|G - G(\phi))\|^2_{L^2(\Omega)} + 4\delta_2 \|\bfB-\bfB^{\eps}\|^2_{L^2(\Omega)^{d\times d}}+\frac{1}{2}\|D_t^{\eps} G(\phi)\|_{L^\infty(\Omega)} \|\bfB-\bfB^{\eps}\|^2_{L^2(\Omega)^{{d\times d}}}
    \\
        & + \frac{1}{4 \delta_2} \|G-G(\phi)\|^2_{L^2(\Omega)}\| \nabla \bfB\|^2_{L^{\infty}(\Omega)^{{d\times d\times d}}}\| \bfv\|^2_{L^{\infty}(\Omega)^d} 
    +\frac{G_S^2}{2\delta_3}  \| \nabla\bfB \|^2_{L^{\infty}(\Omega)^{{d\times d\times d}}} \| \bfB-\bfB^{\eps}\|^2_{L^2(\Omega)^{d\times d}}  
    \\&
    + \delta_3 \| \bfv- \bfv^{\eps}\|_{L^{2}(\Omega)^d}
    + \frac{1}{\delta_2}\|\bb\|^2_{L^{\infty}(\Omega)^{d\times d}}\|G -G(\phi) \|^2_{L^2(\Omega)}
     \\&+\Big (\delta_2\|\nabla\bfv\|^2_{L^{\infty}(\Omega)^{d\times d}}
+ 2 \|G(\phi)\|^2_{L^{\infty}}  \| \nabla \bfv^{\eps} \|^2_{L^{\infty}(\Omega)^{d\times d}} 
+\frac{C_K^2}{ \delta_1}\|G(\phi)\|^2_{L^{\infty}(\Omega)} \|\bb \|^2_{L^{\infty}(\Omega)^{d\times d}} \Big ) \|\bb-\bb^{\eps}\|^2_{L^2(\Omega)^{d\times d}}
\\&
+ \frac{\alpha_F^2}{16 \delta_2}  \| \bb \|^2_{L^{\infty}(\Omega)^{d\times d}} \|\mathbbm{1}_F - \phi\|^2_{L^2(\Omega)}
+\frac{\alpha_F^2}{16 \delta_2}\|\mathbbm{1}_F - \phi\|^2_{L^2(\Omega)} 
+\gamma^2\eps^2\frac{C_K^2}{4 \delta_1}\|\nabla\phi\|^4_{{L^4(\Omega)^d}}.
\end{align*}  
Using the introduced notation for the error energy $\cE(t)$ and collecting the terms, we can rewrite the above estimate as
\begin{align*}
  &    \frac{\D}{\dt} \cE(t)+
  \int_\Omega 2\mu(\phi)|\bD(\bfv - \bfv^\eps)|^2
  + \int_\Omega \alpha(\phi)|\bfB-\bfB^{\eps}|^2 
  \\
&  \leq \Lambda(t)\cE(t)+ F\|\phi-(\mathbbm{1}_F - \mathbbm{1}_S)(t)\|^2_{L^2(\Omega)}+ \gamma^2\eps^2\frac{C_K^2}{\delta_1}\|\nabla\phi(t)\|^4_{{L^4(\Omega)^d}}
  +5\delta_1\|\bD(\bfv-\bfv^\eps)\|^2_{L^2(\Omega)^{{d\times d}}}.
\end{align*}
By taking $\delta_1<  \frac{4 \mu_*}{5}$ and applying the Gr\"onwall lemma, the modeling error estimate \eqref{modelling-error-estimate} follows.
\end{proof}

\newcommand{\dtee}{\mathrm{d}_t}

\section{Numerical method}

To solve the diffuse interface FSI problem~\eqref{m1}-\eqref{m8}, we propose a partitioned numerical method where the problem equations are decoupled so that the Cahn-Hilliard problem, the Navier-Stokes problem, and the transport of the Cauchy-Green stress tensor are solved separately. As it was mentioned in~\cite{Mokbel2018}, the addition of the trace of the Cauchy-Green stress tensor in~\eqref{m4} is necessary for the analysis, however, it also leads to inaccurate description of the interface layer. Hence, this term is omitted in the numerical method.  Since we consider the numerical approximation of only the diffuse interface model, we omit the superscript $^{\eps}$ in the following sections to simplify the notation.
 The discretization in time is done using the finite element method. Because the spatial discretization is standard, below we present the proposed method semi-discretized in time.  
 
 \subsection{Temporal discretization}

Let $\Delta t$ be the  time step and $t^{n} = n \Delta t$, for  all $0 \leq n \leq N$, where the final time is $T=N \Delta t$, and $t^{n+\frac12} = t^n +\frac12 \Delta t,$  for all $n \geq 0$. Let $z^n$ be the approximation of a time-dependent function $z$ at time $n$. We denote the discrete time derivative by 
$$d_t z^{n + 1} := \frac{z^{n + 1} - z^n}{\Delta t}.$$ 
Let $W'_{lin}(\phi^*, \phi^{n+1})$ denote a linearization of  the energy density $W'(\phi^{n+1})$ such that $W'_{lin}(\phi^{n+1}, \phi^{n+1}) = W'(\phi^{n+1})$.
To solve problem~\eqref{m1}--\eqref{m8}, we propose the following algorithm.
\vskip 0.1 in
\noindent\textbf{Algorithm 1}. Given the initial conditions $\bfv^0, \bfB^0$ and $\phi^0$, compute the following steps.
\\
Set the initial guesses as the linearly extrapolated values:
\begin{align*}
& \bfv^{n+\frac12}_{(0)}
=  \frac32 \bfv^{n}  - \frac12 \bfv^{n-1} ,
\\
&\phi^{n+\frac12}_{(0)}
= \frac32 \phi^{n}  - \frac12 \phi^{n-1},
\\
&\bfB^{n+\frac12}_{(0)}
= \frac32  \bfB^{n}  - \frac12 \bfB^{n-1}.
\end{align*}
\noindent\textsc{Step 1}: For  $\kappa\geq 0$, compute until convergence the following  partitioned problem:

\noindent\textbf{Cahn-Hilliard problem:} Find $(\phi^{n+\frac12}_{(\kappa+1)}, m^{n+\frac12}_{(\kappa+1)})$ such that
    \begin{align*}
     &   \frac{\phi^{n+\frac12}_{(\kappa+1)} - \phi^n}{  \Delta t/2} 
        + (\bfv^{n+\frac12}_{(\kappa)} \cdot \nabla) \phi^{n+\frac12}_{(\kappa+1)} = \nabla \cdot \left( M(\phi_{(\kappa)}^{n + \frac12}) \nabla m^{n + \frac12}_{(\kappa+1)} \right) &\text{ in }\Omega,
        \\
 &        \gamma \left( \frac{1}{\epsilon} W'_{lin}(\phi^{n+\frac12}_{(\kappa)},\phi^{n+\frac12}_{(\kappa+1)}) - \eps \Delta \phi^{n + \frac12}_{(\kappa+1)} \right) 
 = m^{n + \frac12}_{(\kappa+1)} & \text{ in }\Omega.
\end{align*}

\noindent\textbf{Transport of the left Cauchy–Green deformation tensor:} Find $\bfB^{n+\frac12}_{(\kappa+1)}$ such that
\begin{align*}
        &G(\phi^{n+\frac12}_{(\kappa+1)})\left( \frac{ \boldsymbol B^{n+\frac12}_{(\kappa+1)} - \boldsymbol B^{n}}{  \Delta t/2} 
        + (\bfv^{n+\frac12}_{(\kappa)} \cdot \nabla ) \bb^{n+\frac12}_{(\kappa+1)}
        - (\nabla \bfv^{n+\frac12}_{(\kappa)}) \bfB^{n+\frac12}_{(\kappa+1)}
        - \bfB^{n+\frac12}_{(\kappa+1)} (\nabla \bfv^{n+\frac12}_{(\kappa)})^{\top} \right)
        \\
     &   \qquad
        + \alpha(\phi^{n+\frac12}_{(\kappa+1)}) (\bfB^{n+\frac12}_{(\kappa+1)}-   \mathbf{I}) 
       + \delta_{stab} \Delta \bfB^{n+{\frac12}}_{(\kappa+1)}
         = 0  & \text{ in }\Omega.
\end{align*}

\noindent\textbf{Navier-Stokes problem:} Find $(\bfv^{n+\frac12}_{(\kappa+1)}, p^{n+\frac12}_{(\kappa+1)})$ such that
\begin{align*}
 	& \frac{\rho(\phi^{n+\frac12}_{(\kappa+1)})\bfv^{n+\frac12}_{(\kappa+1)} -\rho(\phi^n) \bfv^{n}}{  \Delta t/2}
        + (\bfv^{n+\frac12}_{(\kappa)} \cdot \nabla ) (\rho(\phi^{n + \frac12}_{(\kappa+1)}) \bfv^{n+\frac12}_{(\kappa+1)} )
        \\
& \qquad 
        + \frac{\rho_F - \rho_S}{2}M(\phi^{n+\frac12}_{(\kappa+1)})(\nabla m^{n+\frac12}_{(\kappa+1)} \cdot\nabla)\bfv^{n+\frac12}_{(\kappa+1)}
        - \nabla \cdot \left( 2 \mu(\phi^{n+\frac12}_{(\kappa+1)}) \bD(\bfv^{n+\frac12}_{(\kappa+1)}) \right)
        + \nabla p^{n+\frac12}_{(\kappa+1)}
        \\
        &\qquad
        = \nabla \cdot \left(G(\phi^{n+\frac12}_{(\kappa+1)}) \bfB^{n+\frac12}_{(\kappa+1)} \right)
        - \gamma \eps \nabla \cdot \left( \nabla \phi^{n+\frac12}_{(\kappa+1)} \otimes \nabla \phi^{n+\frac12}_{(\kappa+1)} \right)
        + 
        \rho(\phi^{n+\frac12}_{(\kappa+1)})\bff & \text{ in }\Omega,
        \\
        &\nabla \cdot \bfv^{n+\frac12}_{(\kappa+1)}= 0 & \text{ in }\Omega.
\end{align*}

\noindent \textsc{Step 2:}
Denote the converged solutions by 
${\boldsymbol v}^{n+\frac12} , p^{n+\frac12}, 
\bfB^{n+\frac12},
\psi^{n+\frac12},$
 and  
$m^{n+\frac12}.$
Compute the solutions at time $t^{n+1}$ as the following linear extrapolations: 
\begin{align}
&
\rho(\phi^{n+1}){\boldsymbol v}^{n+1} 
= 
2 \rho(\phi^{n+\frac12}) {\boldsymbol v}^{n+\frac12} -  \rho(\phi^{n}){\boldsymbol v}^{n}
&
\mbox{ in }\Omega,
\label{ex1}
\\
& \bfB^{n+1} 
= 2  \bfB^{n+\frac12} -   \bfB^{n} 
&
\mbox{ in }\Omega,
\label{ex2}
\\
&
{\phi}^{n+1} 
= 
2 {\phi}^{n+\frac12} - {\phi}^{n}
&
\mbox{ in }\Omega.
\label{ex3}
\end{align}
Set $n=n+1$, and go back to Step 1.

\begin{remark}
We note that  the following stabilization term is added in the equation for the transport of the left Cauchy–Green deformation tensor:
$$
\delta_{stab} \Delta \bfB^{n+{\frac12}}_{(\kappa+1)}.
$$
This term is commonly added to stabilize the transport problem, and to aid in the analysis~\cite{Mokbel2018}.
\end{remark}

\section{Numerical results}\label{sec:results}

The method proposed in Algorithm~1 is discretized in space using the finite element method and implemented in the high-performance programming language Julia, using the FEM package Gridap~\cite{verdugo2022software,Badia2020}.  In this section, we use the double-well potential defined as  $W'(\phi) = \phi^3 - \phi$, linearized as $$W'_{lin}(\phi_{(\kappa)}^{n+\theta},\phi_{(\kappa+1)}^{n+\theta})=  \phi^{n+\theta}_{(\kappa+1)} (\phi^{n+\theta}_{(\kappa)})^2 - \phi^{n+\theta}_{(\kappa+1)}. $$

We first present an example where we study the convergence rates of the proposed method. Then, we look at an FSI problem with contact, describing an elastic ball falling and reaching the bottom wall.

\subsection{Example 1: Rates of convergence}\label{sub:rates}

To test the convergence rates, we use the method of manufactured solutions. We set the computational domain to be the unit square. The exact solution is given by:
\begin{align*}
        &\bfv_{ref} = 0.2\sin(\pi t)
        \begin{bmatrix}
 \sin(\pi x) \cos(\pi y) \\
   -\cos(\pi x) \sin(\pi y)
        \end{bmatrix}, 
        \\
        &\bfB_{ref} =  10^{-5} \sin(\pi t) (x+10)^2 (y+10)^2
        \begin{bmatrix}
 1 & 6   \\
6  & 5
        \end{bmatrix}, 
        \\
        &\phi_{ref} = \sin(\pi t) \cos(\pi x) \cos(\pi y),   
    \end{align*}
 and $p_{ref}$ is assumed to be a constant.
The forcing term $\boldsymbol{f}$ in~\eqref{m1} is computed using the exact solution. Forcing terms based on the exact solutions are also added to equations~\eqref{m3} and~\eqref{m5}. Dirichlet boundary conditions are imposed for the velocity, and the homogeneous Neumann conditions are used in the Cahn-Hilliard problem. The final time is $T = 0.8$ s. We use $\mathbb{P}_2 - \mathbb{P}_1$ elements for the $\bfv$ and $p$, $\mathbb{P}_1$ elements for $\bfB$, and $\mathbb{P}_2$ elements for $\phi$ and $m$. 

We consider two sets of parameters, shown in Table~\ref{parameters1}. In Set 1, $G_F$ is a positive, nonzero number and the structure is viscoelastic, while in the other set  both $G_F$ and $\mu_S$ are set to zero. In both cases, we use $M(\phi)=1$. 
\begin{table}[ht]
\centering{
        \begin{tabular}{r l | r l }
         \multicolumn{2}{c}{\textbf{Case 1}}    &  \multicolumn{2}{c}{\textbf{Case 2}}  \\
        \textbf{Parameter} & \textbf{Value} & \textbf{Parameter} & \textbf{Value}  \\
        \hline
        \hline
         $\rho_F$  & $1$ &     $\rho_F$  & $1$  \\  
         $\rho_S$  & $1$ &     $\rho_S$ & $1$ \\
         $\mu_F$  & $1$ &   $\mu_F$ & $1$  \\
         $\mu_S$   & $0.5$  &     $\mu_S$   & $0$  \\
          $G_F$  & $0.5$   &   $G_F$    & $0$    \\
         $G_S$     & $1$ &   $G_S$    & $1$\\
         $\alpha_F$   & 1 &   $\alpha_F$  & 1    \\
         $\gamma$ & $10^{-4}$ &   $\gamma$ &  $10^{-4}$  \\ 
    \end{tabular}
    \caption{The parameters used in Example 1.}\label{parameters1}
    }
\end{table}

To compute the rates of convergence, we use the following set of discretization parameters
$$\left\{  \Delta t, \Delta x \right\} = \left\{  \frac{ 0.2}{2^i}, \frac{0.2}{2^i} \right\}_{i=0}^3.$$  
 The parameter $\epsilon$ is defined as $\epsilon = 4 \Delta x$, and it  changes with the mesh size.
We define the relative errors  as
\begin{align*}
        e_{v} = \frac{\|\bfv_{ref}^N - \bfv_h^N \|_{L^2(\Omega)^2}}{\| \bfv^N_{ref} \|_{L^2(\Omega)^2}},
        \quad 
        e_{B} = \frac{\| \bfB^N_{ref} - \bfB^N_h \|_{L^2(\Omega)^{{d\times d}}}}{\| \bfB^N_{ref} \|_{{L^2(\Omega)^{d\times d}}}},
        \quad
 e_{\phi} = \frac{\|\phi^N_{ref} - \phi^N_h \|_{L^2(\Omega)}}{\| \phi^N_{ref} \|_{L^2(\Omega)}},
    \end{align*}
    evaluated at the final time. 
Figure~\ref{errors} shows the relative errors obtained using  the two cases of the problem parameters defined in Table~\ref{parameters1}.
\begin{figure}[H]
    \centering
    \includegraphics[width = 1.0\textwidth]{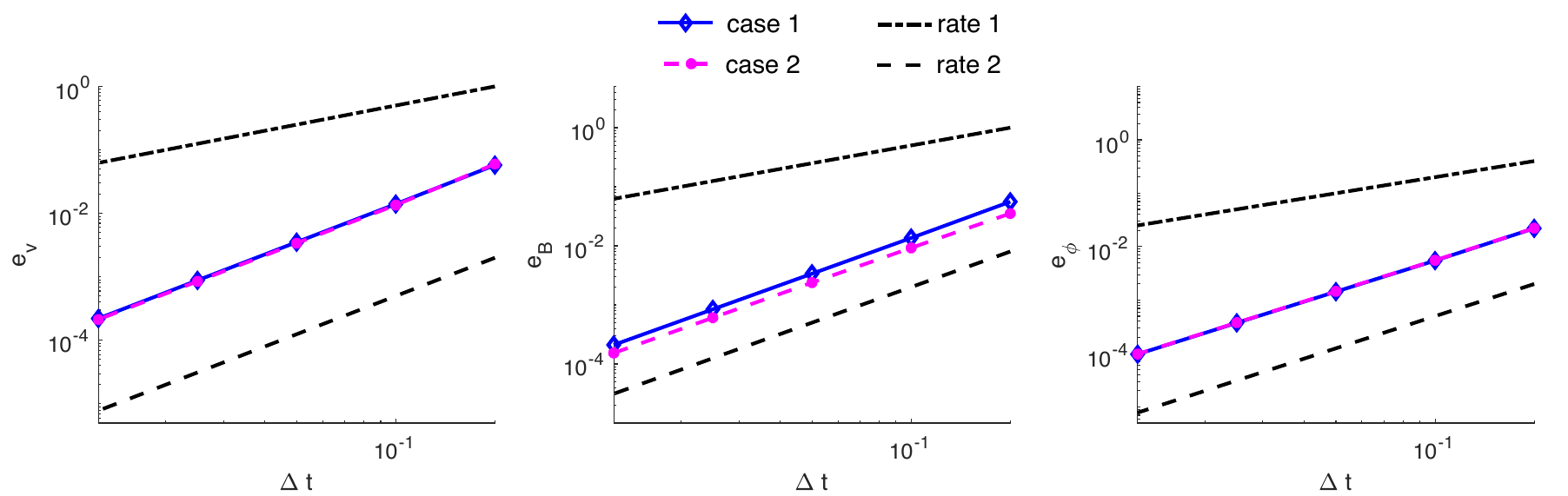}
    \caption{Relative errors computed at the final time.}
    \label{errors}
\end{figure}

As expected, the second order convergence is obtained. The magnitude of the errors is similar for both parameter regimes. However, we note that the stabilization is only needed for Case 2, where we took  $\delta_{stab}=10^{-3},$ while $\delta_{stab}=0$ for Case 1.  To understand why the stabilization is need in Case 2, we considered separate cases when either $\mu_S=0$ or $G_F=0,$ and our results revealed that, in this example, we need stabilization only when  $\mu_S=0$, i.e., in case of purely elastic structures.

We also measure the norms of the solution which we assumed were bounded in Section~\ref{modelingerror}. In particular, $\| \nabla \bfv \|_{L^2(\infty) (\Omega)^{d \times d}}, \| D_t^{\eps} \phi \|_{L^{\infty} (\Omega)},$ and $\| \phi\|^4_{W^{1,4}(\Omega)}$ are shown in Figure~\ref{norms}. The figures are obtained using  both sets of parameters defined in Table~\ref{parameters1}. 
\begin{figure}[H]
    \centering
    \includegraphics[width = 1.0\textwidth]{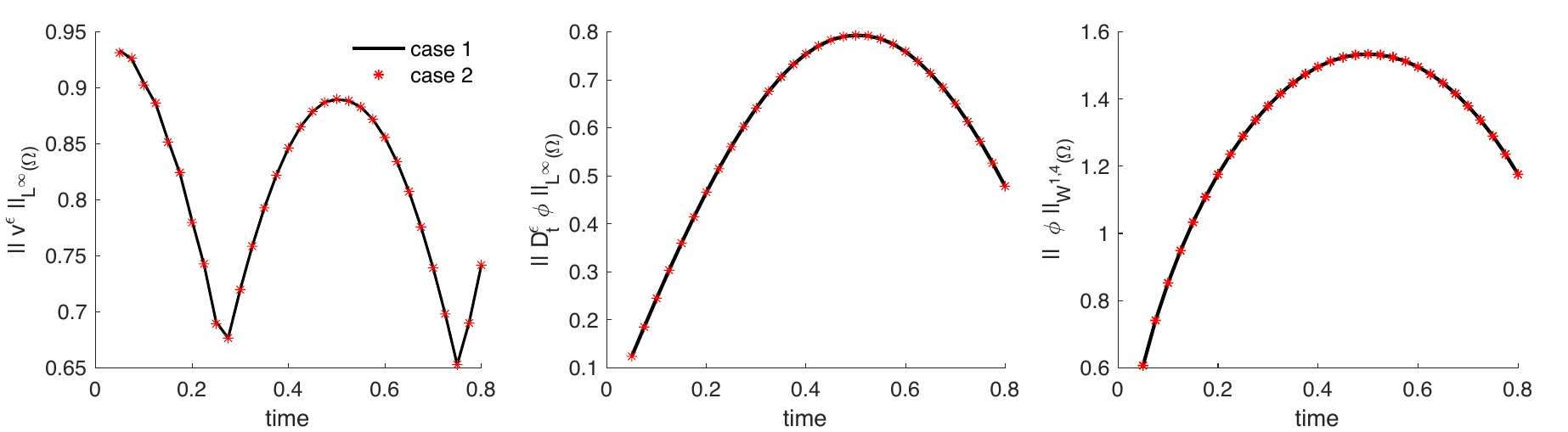}
    \caption{Norms of the solution used in Section~\ref{modelingerror}.}
    \label{norms}
\end{figure}
In all cases the norms  used in the assumptions remain bounded, and no differences are observed when different parameters are used.

\subsection{Example 2}

In the second example, we illustrate the performance of the proposed method on a benchmark problem with contact. We consider an elastic ball immersed in a fluid, falling from a set height and touching the bottom wall. The computational domain is defined as  the unit square, $\Omega=(0,1) \times (0,1).$ To model the contact with a rigid wall using the phase-field approach, it is enough to impose the no-wetting condition on the phase-field function:
$$ \phi = 1\quad  \textrm{on } \partial \Omega \times (0,T).$$
For the fluid velocity, we impose $\bfv =0 $ on $\partial \Omega $, and the homogeneous Neumann conditions are imposed for the transport of $\bfB$. We consider two different cases, one where a ball pulled down by gravity bounces off the bottom wall (Case 1), and one where a softer ball is pulled by a larger force which tightly holds it at the bottom boundary (Case 2).  The parameters used in both cases are given in Table~\ref{parameters2}. 
\begin{table}[ht]
\centering{
    \begin{tabular}{l l | l l }
         \multicolumn{2}{c}{\textbf{Case 1}}    &  \multicolumn{2}{c}{\textbf{Case 2}}  \\
        \textbf{Parameters} & \textbf{Values} & \textbf{Parameters} & \textbf{Values}  \\
        \hline
        \hline
         $\rho_F$ (g/cm$^3$)& $1$ &     $\rho_F$ (g/cm$^3$)& $1$  \\  
         $\rho_S$ (g/cm$^3$)& $10$ &     $\rho_S$ (g/cm$^3$)& $10$ \\
         $\mu_F$ (poise) & $5 \cdot 10^{-4}$ &   $\mu_F$ (poise) & $0.04$  \\
         $\mu_S$ (poise) & $200$  &     $\mu_S$ (poise) & $100$  \\
          $G_F$ (dyne/cm$^2$)  & $0.0$   &   $G_F$ (dyne/cm$^2$)  & $0$    \\
         $G_S$ (dyne/cm$^2$)    & $5 \cdot 10^5$ &   $G_S$ (dyne/cm$^2$)  & $5 \cdot 10^3$\\
         $\alpha_F$  (dyne/cm$^2$s) &$ 5 \cdot 10^4$ &   $\alpha_F$ (dyne/cm$^2$s) & $ 5 \cdot 10^4$    \\
          $\boldsymbol{f}$  (cm/s$^2$) & $(0,-10^3)^{\top}$ &   $\boldsymbol{f}$  (cm/s$^2$) & $(0,-5 \cdot 10^3)^{\top}$ \\ 
    \end{tabular}
    \caption{The parameters used in Example 2.}\label{parameters2}
    }
\end{table}
In both cases, we use $\gamma=10^{-3}, \epsilon=2.5 \cdot 10^{-3},$ and $M=10^{-2}$ in the Cahn-Hilliard problem. In Case 1,  the stabilization coefficient is set to $\delta_{stab}=10^{-3}$, while in Case 2, a larger stabilization of $\delta_{stab}=10^{-1}$ is needed. In both cases, the fluid is initially at rest, $\bfv=0$, and $\bfB= \mathbf{I}$. Initially, the structure is a circle of radius 0.2, centered at $(0.5, 0.7)$. Therefore, at $t=0$, the phase-field function is  defined as
\begin{align*}
\phi = \begin{cases}
-1, & \textrm {if } (x-0.5)^2+(y-0.7)^2 \leq 0.04
\\
1, & \textrm{otherwise}.
\end{cases}
\end{align*}
We use a Union Jack type of mesh consisting of 5204 elements, which is refined around the interface, i.e., in the region where $\nabla \phi$ is non-zero. The simulations are performed until the steady state is reached using a time step of $\Delta t = 10^{-4}$.

\begin{figure}[ht]
    \centering
    \includegraphics[width = 0.8\textwidth]{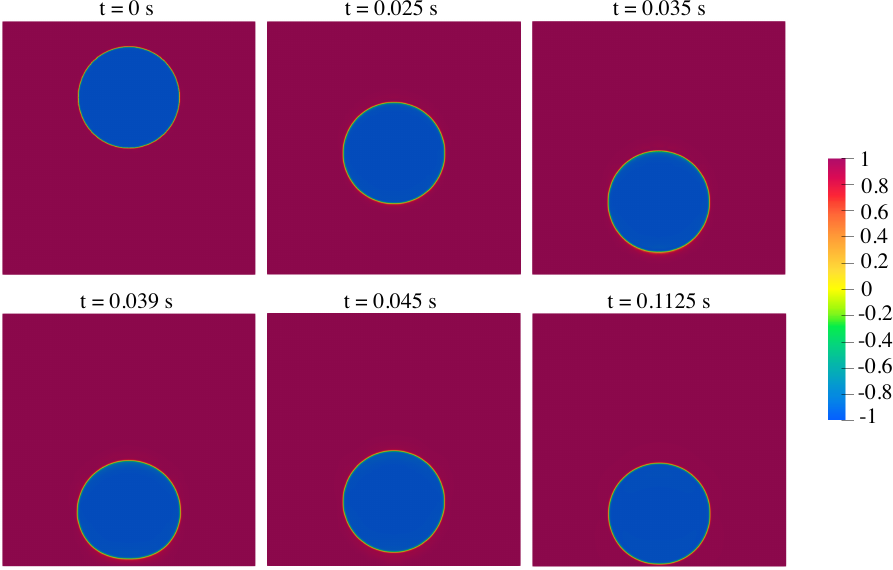}
    \caption{The phase-field function, $\phi$, indicating the position of the elastic structure in Case 1.}
    \label{case1phi}
\end{figure}
Figure~\ref{case1phi} shows the  phase-field function for Case 1, indicating the position of the elastic ball over time. Around $t=0.039$ s, the ball makes the contact with the bottom wall. Afterwards, it slightly bounces back up, and then finally settles back down at the bottom of the domain. 
The velocity streamlines at three different time instances, colored by the velocity magnitude, are shown in Figure~\ref{case1velocity}. 
\begin{figure}[ht]
    \centering
    \includegraphics[width = 0.8\textwidth]{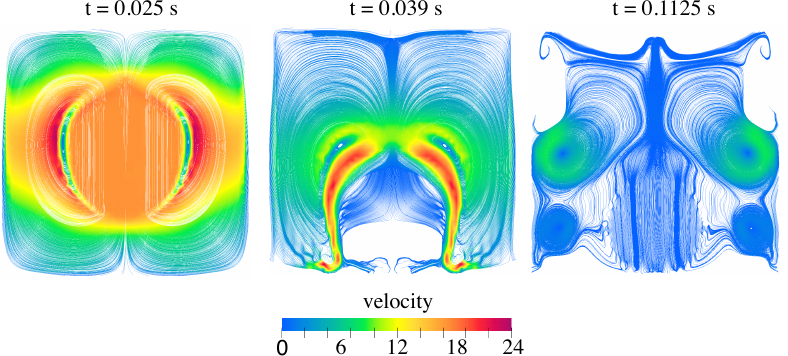}
    \caption{The velocity streamlines colored by the velocity magnitude in Case 1.}
    \label{case1velocity}
\end{figure}

In Case 2, the parameters are chosen so that the ball is softer, and a larger force is pulling it down. This results in a different dynamics compared to Case 1. 
\begin{figure}[ht]
    \centering
    \includegraphics[width = 0.8\textwidth]{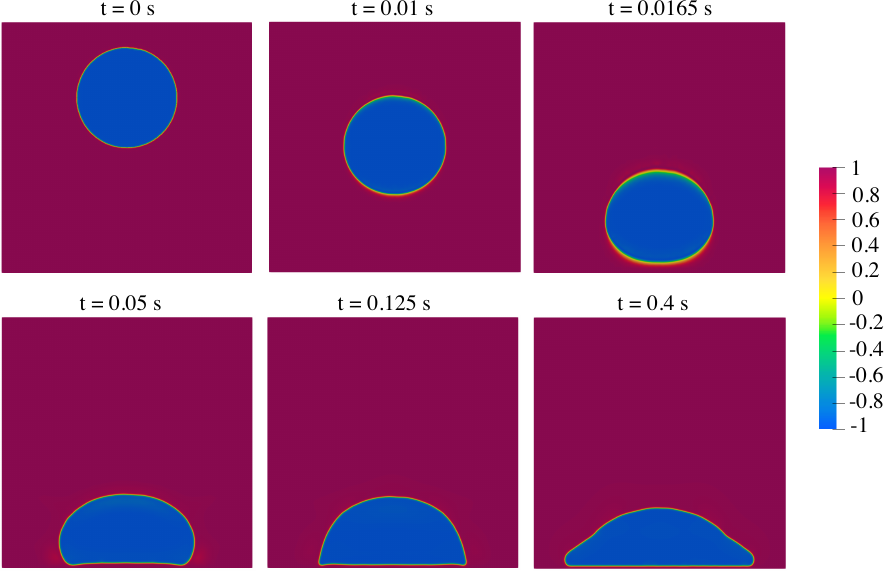}
    \caption{The phase-field function, $\phi$, indicating the position of the elastic structure in Case 2.}
    \label{case2phi}
\end{figure}
The phase-field function obtained using this set of parameters is shown in Figure~\ref{case2phi}. Due to a larger force, the ball falls faster,  and undergoes a larger deformation compared to the previous case. While the shape of the ball adjusts due to a recoil after the contact, the ball never detaches from the bottom wall. Instead, the force slowly pulls it further down. The velocity streamlines in this case are shown in Figure~\ref{case2velocity}. 
\begin{figure}[ht]
    \centering
    \includegraphics[width = 0.8\textwidth]{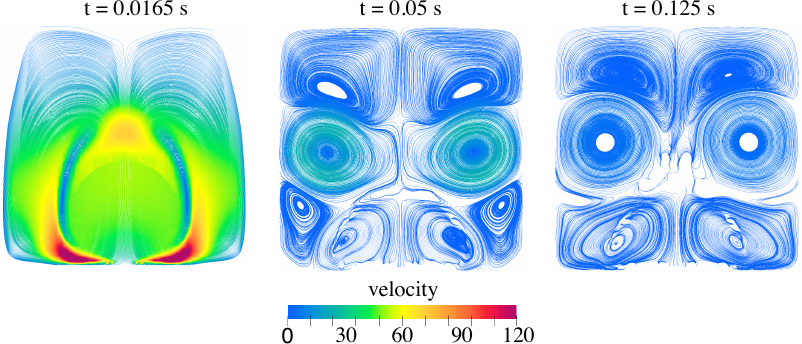}
    \caption{The velocity streamlines colored by the velocity magnitude in Case 2.}
    \label{case2velocity}
\end{figure}

\section{Conclusion}

In this work, we have studied the diffuse interface method for FSI problems formulated in the Eulerian framework. We derived energy estimates for the diffuse interface model at the continuous level, and obtained error estimates quantifying the modeling error between the diffuse and sharp interface solutions. 
We further proposed a novel numerical scheme for the coupled system. The scheme was validated on two numerical examples. In the first example, we performed convergence tests demonstrating that the method achieves second-order accuracy. In the second example, we applied  the numerical  scheme to simulate the contact between an elastic ball immersed in a fluid and a rigid wall, where contact was handled in a natural and straightforward manner by imposing a no-wetting boundary condition on the phase-field function. A couple of parameter regimes were explored, revealing different problem dynamics.
Our results demonstrate that the diffuse interface approach offers a flexible  framework for FSI, with particular advantages in handling topological events such as contact. In the future work, we  may consider extensions to three-dimensional settings  and multibody contact problems.

\section{Acknowledgement}
The authors would like to acknowledge Jordi Manyer for his assistance in learning Gridap, and for answering our  numerous questions throughout this work.

\section{Data Availability}

Enquiries about data availability should be directed to the authors.

\section{Declarations}
MB is  supported in part by the National Science Foundation via grants NSF DMS-2208219 and NSF DMS-2205695. BM was supported by the Croatian Science Foundation under the project number IP-2022-10-2962 and by Croatia-USA bilateral grant “The mathematical framework for the diffuse interface method applied to coupled problems in fluid dynamics”. FA is supported by a Society of Science Postdoctoral Fellowship from the College of Science at the University of Notre Dame.

\bibliographystyle{siam}
\bibliography{bib}

\end{document}